\documentclass[12pt]{amsart}

\usepackage{amssymb}

\usepackage{enumitem}

\makeatletter
\@namedef{subjclassname@2020}{%
  \textup{2020} Mathematics Subject Classification}
\makeatother

\usepackage[T1]{fontenc}

\newtheorem{theorem}{Theorem}[section]
\newtheorem{corollary}[theorem]{Corollary}
\newtheorem{lemma}[theorem]{Lemma}

\theoremstyle{definition}
\newtheorem{definition}[theorem]{Definition}
\newtheorem{remark}[theorem]{Remark}
\newtheorem{example}[theorem]{Example}

\numberwithin{equation}{section}

\usepackage{amsfonts, amsthm, amsmath, amssymb}
\usepackage{hyperref}
\hypersetup{colorlinks=false}

\usepackage{enumitem}
\usepackage{textcomp}
\usepackage[T1]{fontenc}
\usepackage{url}

\usepackage[margin=1.5in]{geometry}

\RequirePackage{mathrsfs} \let\mathcal\mathscr
\numberwithin{equation}{section}

\usepackage[capitalise, noabbrev]{cleveref}
\usepackage{colonequals}

\newcommand{\Mod}[1]{\ (\mathrm{mod}\ #1)}
\newcommand{\op}[1]{\operatorname{#1}}
\newcommand{\mc}[1]{\mathcal{#1}}

\newcommand{\prim}{\textrm{prim}}

\renewcommand{\phi}{\varphi}
\renewcommand{\rho}{\varrho}

\renewcommand{\P}{\mathbb{P}}
\newcommand{\Proj}{\P}

\newcommand{\A}{\mathbb{A}}

\newcommand{\Z}{\mathbb{Z}}

\newcommand{\N}{\mathbb{N}}
\newcommand{\Q}{\mathbb{Q}}
\newcommand{\R}{\mathbb{R}}

\newcommand{\OO}{\mathcal{O}}

\newcommand{\p}{\mathfrak{p}}

\newcommand{\isom}{\cong}

\renewcommand{\leq}{\leqslant}
\renewcommand{\geq}{\geqslant}

\newcommand{\bm}[1]{\mathbf{#1}}

\newcommand{\bx}{\mathbf{x}}
\newcommand{\by}{\mathbf{y}}

\newcommand{\bz}{\mathbf{z}}
\newcommand{\bw}{\mathbf{w}}

\newcommand{\eps}{\epsilon}
\newcommand{\bmep}{\bm{\eps}}
\newcommand{\bmeps}{\bm{\eps}}
\newcommand{\al}{\alpha}

\DeclareMathOperator{\Pic}{Pic}

\newcommand{\Hom}{\op{Hom}}

\newcommand{\supp}{\op{supp}}

\newtheorem{conjecture}[theorem]{Conjecture}

\theoremstyle{definition}

\numberwithin{equation}{section}

\newcommand{\nid}{\noindent}

\newcommand{\ra}{\rightarrow}

\newcommand{\tensor}{\otimes}

\newcommand{\bs}{\backslash}

\newcommand{\PP}{\Proj}
\usepackage{tikz}
\usetikzlibrary{3d,calc}
\usetikzlibrary{positioning}

\begin{document}


\baselineskip=17pt


\title[On the leading constant for Campana points]{On the leading constant in the Manin-type conjecture for Campana points}

\author[A. Shute]{Alec Shute}

\address{Institute of Science and Technology Austria\\ Am Campus 1, 3400 Klosterneuburg, Austria}
\email{alec.shute@ist.ac.at}

\date{}

\begin{abstract}
We compare the Manin-type conjecture for Campana points recently formulated by Pieropan, Smeets, Tanimoto and V\'{a}rilly-Alvarado with an alternative prediction of Browning and Van Valckenborgh in the special case of the orbifold $(\mathbb{P}^1,D)$, where $D = \frac{1}{2}[0]+\frac{1}{2}[1]+\frac{1}{2}[\infty]$. We find that the two predicted leading constants do not agree, and we discuss whether thin sets could explain this discrepancy. Motivated by this, we provide a counterexample to the Manin-type conjecture for Campana points, by considering orbifolds corresponding to squareful values of binary quadratic forms.
\end{abstract}

\subjclass[2020]{11D45 (primary), 14G05 (secondary)}

\keywords{Rational points, Campana orbifolds, squareful numbers}

\maketitle

\section{Introduction}
The study of Campana points is an emerging area of interest in arithmetic geometry as a way to interpolate between rational and integral points. Campana orbifolds, first introduced in \cite{campana2004orbifolds} and \cite{campana2011orbifoldes}, consist of a variety $X$ and a weighted boundary divisor $D$ of $X$. The Campana points associated to the orbifold $(X,D)$ can be viewed as rational points of $X$ that are integral with respect to $D$. In the recent paper \cite{pieropan2019campana}, Pieropan, Smeets, Tanimoto and V\'{a}rilly-Alvarado formulate a Manin-type conjecture for the quantitative study of Campana points on Fano Campana orbifolds, which henceforth we shall refer to as the \textit{PSTV-A conjecture}. The authors establish their conjecture in the special case of vector group compactifications, using the height zeta function method developed by Chambert-Loir and Tschinkel \cite{chambert2002distribution}, \cite{chambert2012integral}. 

The arithmetic study of Campana points is still in its early stages. Initial results in \cite{browning2012sums}, \cite{van2012squareful} and \cite{browning2019arithmetic}, which predate the formulation of the PSTV-A conjecture, concern squareful and $m$-full values of hyperplanes of $\PP^{n+1}$. (We recall that a nonzero integer $z$ is $m$-\textit{full} if for any prime $p$ dividing $z$, we have $p^m\mid z$, and \textit{squareful} if it is $2$-full.) Following discussions in the Spring 2006 MSRI program on rational and integral points on higher dimensional varieties, Poonen \cite{poonen2006projective} posed the problem in  of finding the number of coprime integers $z_0,z_1$ such that $z_0,z_1$ and $z_0+z_1$ are all squareful and bounded by $B$. In the language of the PSTV-A conjecture, this corresponds to counting Campana points on the orbifold $(\PP^1,D)$, where $D$ is the divisor $\frac{1}{2}[0]+\frac{1}{2}[1]+\frac{1}{2}[\infty]$. Upper and lower bounds for this problem were obtained by Browning and Van Valckenborgh \cite{browning2012sums}, but finding an asymptotic formula remains wide open. Van Valckenborgh \cite{van2012squareful} considers a higher-dimensional analogue of this problem by defining a Campana orbifold $(\PP^{n},D)$, where
$$D_i = \begin{cases} \{z_i = 0\},&\textrm{if } 0\leq i \leq n,\\
\{z_0+\cdots + z_{n}= 0\},&\textrm{if } i=n+1.
\end{cases}
$$
Choosing the height $H$ on $\PP^{n}(\Q)$ defined by 
\begin{equation}\label{the height}
H(z) = \max(|z_0|,\ldots, |z_{n}|, |z_0+\cdots +z_{n}|),
\end{equation}
for a representative $(z_0,\ldots, z_{n}) \in \Z^{n+1}_{\prim}$ of $z$, this leads to the counting problem
\begin{equation}\label{counting problem for k squarefuls}
    N_{n}(B) \colonequals \frac{1}{2}\#\left\{(z_0,\ldots, z_{n+1})\in (\Z_{\neq 0})^{n+2}_{\prim}:
\begin{tabular}{l}
$z_0+\cdots + z_{n}=z_{n+1}$,\\
$|z_0|,\ldots, |z_{n+1}| \leq B,$\\ 
$z_0,\ldots, z_{n+1}\textrm{ squareful}$
\end{tabular}
\right\}.
\end{equation}
Van Valckenborgh \cite[Theorem 1.1]{van2012squareful} proves that for any $n\geq 3$, we have $N_n(B) \sim cB^{n/2}$ as $B \ra \infty$, for an explicit constant $c>0$. In \cite{4squareful2021}, we extend the treatment to handle the case $n=2$. Work of Browning and Yamagishi \cite{browning2019arithmetic} considers a more general orbifold $(\PP^n, D)$, where the $D_i$ are as above, and $D = \sum_{i=0}^{n+1}(1-\frac{1}{m_i})D_i$ for integers $m_0, \ldots, m_{n+1} \geq 2$. Their main result is an asymptotic formula for the number of Campana points on this orbifold (with the same height as in (\ref{the height})), under the assumption that there exists some $j\in \{0,\ldots, n+1\}$ such that 
$$\sum_{\substack{0\leq i \leq n+1\\ i \neq j}}\frac{1}{m_i(m_i+1)}\geq 1.$$
 
Following the formulation of the PSTV-A conjecture, several further cases have been treated. Pieropan and Schindler \cite{pieropan2020hyperbola} establish the PSTV-A conjecture for complete smooth split toric varieties satisfying an additional technical assumption, by developing a very general version of the hyperbola method. Xiao \cite{xiao2020campana} treats the case of biequivariant compactifications of the Heisenberg group over $\Q$, using the height zeta function method. Finally, Streeter \cite{streeter2021campana} studies $m$-full values of norm forms by counting Campana points on the orbifold $(\PP_K^{d-1},(1-\frac{1}{m})V(N_{E/K}))$, where $K$ is a number field, $V(N_{E/K})$ is the divisor cut out by a norm form associated to a degree-$d$ Galois extension $E/K$, and $m\geq 2$ is an integer which is coprime to $d$ if $d$ is not prime. 

In \cite{pieropan2019campana}, \cite{pieropan2020hyperbola} and \cite{xiao2020campana}, the leading constants for the counting problems considered were reconciled with the prediction from the PSTV-A conjecture. In the case of Campana points for norm forms, Streeter \cite[Section 7.3]{streeter2021campana} provides an example where the leading constant in \cite[Theorem 1.4]{streeter2021campana} differs from the constant defined in the PSTV-A conjecture. It remains unclear whether this could be explained by the removal of a thin set. For the papers \cite{browning2012sums}, \cite{van2012squareful} and \cite{browning2019arithmetic}, however, no subsequent attempts to compare the leading constants have been made. In this paper, we attempt to remedy this by making a detailed study of the leading constant from \cite{browning2012sums} in the context of the PSTV-A conjecture.

We now summarise the approach employed by Van Valckenborgh in the proof of \cite[Theorem 1.1]{van2012squareful}. We can write each nonzero squareful number $z_i$ uniquely in the form $x_i^2y_i^3$ for a positive integer $x_i$ and a squarefree integer $y_i$. For a fixed choice of $\by = (y_0, \ldots, y_{n+1}) \in (\Z_{\neq 0})^{n+2}_{\prim}$, the equation $z_0+ \cdots + z_{n} = z_{n+1}$ can be viewed as a quadric $Q_{\by}$ in $\PP^{n+1}$ defined by the equation 
$$ y_0^3x_0^2 + \cdots + y_{n}^3x_{n}^2 = y_{n+1}^3x_{n+1}^2.$$
Using the circle method, one can estimate the number $N_{\by}^+(B)$ of rational points $[x_0: \cdots : x_{n+1}]$ on $Q_{\by}$ with $(x_0, \ldots, x_{n+1}) \in (\Z_{\neq 0})^{n+2}_{\prim}$, satisfying the conditions $\gcd(x_0y_0,\ldots,x_{n+1}y_{n+1}) =1$ and $|x_i^2y_i^3| \leq B$ for all $i\in \{0,\ldots, n+1\}$. Now
$$N_{n}(B)=\frac{1}{2^{n+2}}\sum_{\by \in (\Z_{\neq 0})^{n+2}}\mu^2(y_0)\cdots \mu^2(y_{n+1})N^+_{\by}(B),$$
where $\mu$ denotes the M\"{o}bius function. The factor $\frac{1}{2^{n+2}}$ is obtained from the factor $1/2$ in (\ref{counting problem for k squarefuls}), together with the fact that for each $(z_0,\ldots, z_{n+1})\in (\Z_{\neq 0})^{n+2}_{\prim}$, there are $2^{n+1}$ corresponding points $[x_0:\cdots:x_{n+1}]$ enumerated by $N_{\by}^+(B)$, differing only by changes of signs of $x_0,\ldots, x_{n+1}$. One seeks to obtain an asymptotic formula for $N_n(B)$ by getting enough uniformity in the asymptotic formulas for $N_{\by}^+(B)$. 

With this approach, the leading constant for $N_n(B)$ is expressed as an infinite sum of constants $c_{\by}$ arising from Manin's conjecture applied to $N_{\by}^+(B)$. This is the point of view taken in \cite[Section 5]{van2012squareful} for $n\geq 3$, and it is also how we express the leading constant in \cite{4squareful2021} for the case $n=2$. When $n=1$, it leads to the following prediction \cite[Conjecture 1.1]{browning2012sums}.
\begin{conjecture}[Browning, Van Valckenborgh, 2012]\label{conjecture 1.1}
We have
$$N_1(B) \sim 3c_{\textrm{BV}}B^{1/2},$$
where the constant $c_{\textrm{BV}}$ is given explicitly in \cite[(2--12)]{browning2012sums} (and also in (\ref{the second prediction})), and is expressed as a sum over $(y_0,y_1,y_2)$ of constants arising from Manin's conjecture applied to the conics $x_0^2y_0^3+x_1^2y_1^3 = x_2^2y_2^3$. 
\end{conjecture}
The reason for the factor $3$ in Conjecture \ref{conjecture 1.1} is explained in Lemma \ref{count comparison}, and is due to the counting problem considered in \cite{browning2012sums} being over $\N^3_{\prim}$ rather than $(\Z_{\neq 0})^3_{\prim}$. 

By focusing on the contribution to $N_{\by}^+(B)$ from the range $|\by|\leq B^{\theta}$, for a small absolute constant $\theta >0$, it is possible to prove the lower bound
\begin{equation}\label{unconditional surprise}
    N_1(B) \geq 3c_{\textrm{BV}}B^{1/2}(1+o(1)),
\end{equation}
where $c_{\textrm{BV}}$ is as defined in Conjecture \ref{conjecture 1.1}. This is achieved in \cite[Theorem 1.2]{browning2012sums}, where it is also established that $c_{\textrm{BV}}$ takes the numerical value $2.68...$ correct to two decimal digits. 

Throughout this paper, we let $c_{\textrm{PSTV-A}}$ denote the leading constant predicted by the PSTV-A conjecture \cite{pieropan2019campana}. For the orbifold  $(\PP^1,\frac{1}{2}[0]+\frac{1}{2}[1]+\frac{1}{2}[\infty])$ corresponding to the counting problem $N_1(B)$, there does not appear to be any obvious thin set to remove. Therefore, we might naturally expect that $c_{\textrm{PSTV-A}}$ is the leading constant for $N_1(B)$ itself, and consequently, in view of the lower bound in (\ref{unconditional surprise}), that  $c_{\textrm{PSTV-A}} \geq 3c_{\textrm{BV}}$. In Section \ref{countexample}, we shall prove the following result, which shows that in fact, $c_{\textrm{PSTV-A}}< 3c_{\textrm{BV}}$. 
\begin{theorem}\label{campana manin prediction}
For the orbifold corresponding to the counting problem $N_1(B)$, the leading constant predicted by the PSTV-A conjecture is
\begin{equation}\label{cpstva}
c_{\textrm{PSTV-A}}=\frac{9}{2\pi}\prod_p\left(1+\frac{3p^{-3/2}}{1+p^{-1}}\right).
\end{equation}
Moreover, $c_{\textrm{PSTV-A}}/3 = 2.56785632...$, accurate up to eight digits. 
\end{theorem}
We define
\begin{equation}\label{campana points for three squarefuls}
\mathcal{C} =\left\{[z_0:z_1] \in \PP^1(\Q): (z_0,z_1) \in \Z^2_{\prim}, z_0,z_1,z_0+z_1 \textrm{ squareful and nonzero}\right\}
\end{equation}
to be the set of Campana points under consideration. If the PSTV-A conjecture is correct, the discrepancy between $c_{\textrm{PSTV-A}}$ and $3c_{\textrm{BV}}$ could be explained by thin sets in one of the following two ways:
\begin{enumerate}
    \item  The set of Campana points $\mathcal{C}$ is itself thin, a situation which is explicitly excluded in the statement \cite[Conjecture 1.1]{pieropan2019campana} of the PSTV-A conjecture. 
    \item There is a thin set $\mc{T}\subset \mathcal{C}$ of Campana points such that the removal of $\mc{T}$ from the count $N_1(B)$ reduces the leading constant to $c_{\textrm{PSTV-A}}$.  
\end{enumerate}

Recent work of Nakahara and Streeter \cite{samthincampana} tackles the question of when the set of Campana points corresponding to a log Fano orbifold $(\PP^n,D)$ can be a thin set. The authors establish a connection between thin sets of Campana points and weak approximation, in the spirit of Serre's arguments in \cite[Theorem 3.5.7]{serretopicsgalois}. Together with \cite[Corollary 1.4]{samthincampana}, this implies that $\mc{C}$ is not itself thin. It remains to consider whether the second explanation above could hold. In Section \ref{sec:thin sets}, we prove the following result. 
\begin{theorem}\label{you can get anything} Suppose that Conjecture \ref{conjecture 1.1} holds. Let the height function $H$ be as defined in (\ref{the height}) for $n=1$. Then for any real number $\lambda\in(0,3c_{\textrm{BV}}]$, there is a Campana thin subset $\mathcal{T}\subseteq \mathcal{C}$, as defined in \cite[Definition 3.7]{pieropan2019campana}, such that  
$$\#\{z \in \mathcal{C}\bs\mathcal{T}: H(z)\leq B\} \sim \lambda B^{1/2}.$$
\end{theorem}

Theorem \ref{you can get anything} demonstrates that if Conjecture \ref{conjecture 1.1} holds, we can obtain any leading constant in $(0,3c_{\textrm{BV}}]$, including the constant $c_{\textrm{PSTV-A}}$, by the removal of an appropriate thin set. From this point of view, the PSTV-A conjecture as stated in \cite{pieropan2019campana} seems somewhat unsatisfactory, in that all points can lie on accumulating thin subsets. However, there does not appear to be any thin set with a clear geometric meaning which we can remove in order to obtain the constant $c_{\textrm{PSTV-A}}$, and so currently $3c_{\textrm{BV}}$ seems the most natural prediction for the leading constant in this example.

\begin{remark} We have considered $N_1(B)$ for simplicity, but it seems likely that similar statements hold for $N_n(B)$ with $n \geq 2$. In these cases, as mentioned above, we know that the analogue of Conjecture \ref{conjecture 1.1} holds, and so we should be able to obtain unconditional analogues of Theorem \ref{you can get anything} for any $n\geq 2$.   
\end{remark}

Motivated by the above example, in Section \ref{section: squareful values of binary quadratic forms} we carry out a similar comparison in the case of squareful values of a binary quadratic form. For fixed positive, squarefree and coprime integers $a,b$ satisfying $a,b \equiv 1 \Mod{4}$, we consider the counting problem 
$$ N(B)= \frac{1}{2}\#\left\{(x,y) \in \Z^2_{\prim}: |x|,|y|\leq B, ax^2+by^2 \textrm{ squareful}\right\}.$$
This corresponds to the Campana orbifold $(X,D)=(\PP^1,\frac{1}{2}V(ax^2+by^2))$ over $\Q$, together with the obvious $\Z$-model $(\mc{X},\mc{D})$, and the height $H$ on $\PP^1(\Q)$ given by $H([x:y])=\max(|x|,|y|)$ for $(x,y)\in \Z^2_{\prim}$. By \cite[Theorem 1.1]{samthincampana} and \cite[Proposition 3.15]{samthincampana}, the set of Campana points in this example is not itself thin. In Theorem \ref{binary form campana constant}, we compute the constant $c_{\textrm{PSTV-A}}$ for this example.  

In Section \ref{section: squareful values of binary quadratic forms}, we also prove the following theorem, which can be thought of as an unconditional analogue of Conjecture \ref{conjecture 1.1} for the counting problem $N(B)$.
\begin{theorem}\label{asymptotic for N(B) intro}
For any $\eps>0$, we have $N(B) = cB+O(B^{85/86+\eps})$, where the implied constant depends only on $a,b$ and $\eps$. The leading constant $c$ is given explicitly in (\ref{manin constant for norm form}) as a sum over $v$ of constants arising from Manin's conjecture applied to the conics $ax^2+by^2 = u^2v^3$. 
\end{theorem}

\begin{remark} When $a=1$, $N(B)$ counts squareful values of the norm form $x^2+by^2$. This is a very special case of a result by Streeter \cite[Theorem 1.4]{streeter2021campana}. The constant from \cite[Theorem 1.4]{streeter2021campana} and the constant $c$ from Theorem \ref{asymptotic for N(B) intro} must therefore agree. However, the proof of \cite[Theorem 1.4]{streeter2021campana} proceeds via very different methods, using height zeta functions and Fourier analysis, leading to a constant that involves a sum of limits of global Fourier transforms of $2$-torsion toric characters.
\end{remark}

The constants $c$ and $c_{\textrm{PSTV-A}}$ are often not equal. In the norm form case $a=1$, we show that $c_{\textrm{PSTV-A}}<c$ whenever $b>1$. Analogously to Theorem \ref{you can get anything}, any constant in $(0,c]$ could be obtained by the removal of an appropriate thin set. When $a,b>1$, however, we shall show that sometimes $c<c_{\textrm{PSTV-A}}$. The significance of this is that thin sets cannot explain the discrepancy between the constants. Thus Theorem \ref{asymptotic for N(B) intro} provides the basis for the following counterexample to the leading constant predicted by the PSTV-A conjecture. 

\begin{corollary}\label{binary quadratic counterexample}
Let $a=37$ and $b=109$. Then the PSTV-A conjecture does not hold for the orbifold $(\mc{X},\mc{D})$ and the height $H$ defined above. 
\end{corollary}

\section{The Manin-type conjecture for Campana points}
In this section, we recall from \cite{pieropan2019campana} the definition of Campana points and the statement of the PSTV-A conjecture. Throughout this section, we work over a number field $K$. 

\begin{definition} A \textit{Campana orbifold} is a pair $(X,D)$, where $X$ is a smooth variety over $K$ and 
$$D = \sum_{\alpha \in \mathcal{A}}\eps_{\alpha}D_{\alpha}$$
is an effective Weil $\Q$-divisor of $X$ over $K$ (where the $D_{\alpha}$ are prime divisors) such that
\begin{enumerate}
    \item For all $\alpha \in \mathcal{A}$, either $\eps_{\alpha} =1$ or $\eps_{\alpha}$ takes the form $1-1/m_{\alpha}$ for some $m_{\alpha}\in \Z_{\geq 2}$.
\item The support $D_{\textrm{red}}=\sum_{\alpha \in \mathcal{A}}D_{\alpha}$ of $D$ has strict normal crossings on $X$.
\end{enumerate}
We say that a Campana orbifold is \textit{klt} if $\eps_{\alpha}\neq 1$ for all $\alpha \in \mathcal{A}$.
\end{definition}

Let $(X,D)$ be a Campana orbifold. Campana points will be defined as points $P\in X(K)$ satisfying certain conditions. These conditions are dependent on a finite set $S$ of places of $K$ containing all archimedean places, and a choice of \textit{good integral model} of $(X,D)$ over $\OO_{K,S}$. This model is defined to be a pair $(\mathcal{X},\mathcal{D})$, where $\mathcal{X}$ is a flat, proper model of $X$ over $\OO_{K,S}$, with $\mathcal{X}$ regular, and 
$$\mathcal{D} = \sum_{\alpha \in \mathcal{A}}\eps_{\alpha}\mathcal{D}_{\alpha},$$
where $\mathcal{D}_{\alpha}$ denotes the Zariski closure of $D_{\alpha}$ in $\mathcal{X}$. 

\begin{definition} Let $P \in (X\bs D_{\textrm{red}})(K)$. For a place $v \notin S$, let $\mathcal{P}_v$ denote the induced point in $\mathcal{X}(\OO_v)$ obtained via the valuative criterion for properness, as stated in \cite[Thm. II.4.7]{hartshorne1977algebraic}. For $\alpha \in \mc{A}$, we define the \textit{intersection multiplicity} $n_v(\mathcal{D}_{\alpha},P)$ of $\mc{D}_{\al}$ and $P$ at $v$ to be the colength of the ideal $\mc{P}_v^{*}\mc{D}_{\al}$ in $\OO_v$. The \textit{intersection number} of $P$ and $\mathcal{D}$ at $v$ is defined to be 
$$n_v(\mathcal{D},P) = \sum_{\alpha \in \mc{A}}\eps_{\alpha}n_v(\mathcal{D}_{\alpha},P).$$
\end{definition}

\begin{definition} Let $(X,D)$ be a Campana orbifold with a good integral model $(\mc{X}, \mc{D})$ over $\OO_{K,S}$. A point $P \in (X\bs D_{\textrm{red}})(K)$ is a \textit{Campana} $\OO_{K,S}$\textit{-point} of $(\mc{X},\mc{D})$ if for all $v \notin S$ and all $\alpha \in \mc{A}$, we have
\begin{enumerate}
    \item If $\eps_{\alpha} = 1$, then $n_v(\mc{D}_{\al},P) =0$.
    \item If $\eps_{\al} \neq 1$, so that $\eps_{\al} = 1-1/m_{\al}$ for some $m_{\al} \in \Z_{\geq 2}$, then either $n_v(\mc{D}_{\al},P) =0$ or $n_v(\mc{D}_{\al},P)\geq m_{\al}$.
\end{enumerate}
We denote the set of Campana $\OO_{K,S}$-points of $(\mc{X},\mc{D})$ by $(\mc{X},\mc{D})(\OO_{K,S})$. 
\end{definition}

\example\label{example of campana orbifolds} When $K=\Q$, Campana points are related to $m$-full values of polynomials. We consider projective space $X= \PP^n$, and a strict normal crossings divisor
$$D = \sum_{i=0}^k \left(1-\frac{1}{m_i}\right)D_i,$$
where $m_i \geq 2$ are integers, and $D_i$ are prime divisors on $X$ defined by irreducible polynomials $f_i$ with integral coefficients. Choosing the obvious good integral model $(\mc{X},\mc{D})$, a rational point $z \in (X\bs\bigcup_{i=0}^k D_i)(\Q)$, represented by $(z_0,\ldots, z_n)\in \Z^{n+1}_{\prim}$, is a Campana $\Z$-point of $(\mc{X},\mc{D})$ if and only if $f_i(z_0, \ldots, z_n)$ is $m_i$-full for all $i\in \{0,\ldots,k\}$. In particular, the Campana points $\mc{C}$ defined in (\ref{campana points for three squarefuls}) fit into this context, by making the choices $X=\PP^1, k=2, m_0=m_1=m_2=2$, and $f_0=z_0,f_1=z_1, f_2 = z_0+z_1$. 

\begin{definition}\label{what are campana thin sets} We recall that for an irreducible variety $X$ over $K$, a subset $A\subset X(K)$ is \textit{type I} if $A=Z(K)$, where $Z$ is a proper closed subvariety of $X$, and \textit{type II} if $A=\varphi(V(K))$, where $V$ is an integral projective variety with $\dim(V) = \dim(X)$ and $\varphi\colon V \ra X $ is a dominant morphism of degree at least 2. A \textit{thin} set of $X(K)$ is a subset of $X(K)$ which is contained in a finite union of type I and type II sets. In \cite[Definition 3.7]{pieropan2019campana}, a thin set of Campana $\OO_{K,S}$-points is defined to be the intersection of a thin set of $X(K)$ with the set of Campana points $(\mathcal{X}, \mathcal{D})(\OO_{K,S})$.
\end{definition}

We now come to the statement of the PSTV-A conjecture given in \cite[Conjecture 1.1]{pieropan2019campana}. Let $(X,D)$ be a Campana orbifold over $K$ with a good integral model $(\mc{X},\mc{D})$ over $\OO_{K,S}$. Let $(\mc{L},\|\cdot\|)$ be an adelically metrized big and nef line bundle on $X$ with associated divisor class $[L]$. Let $H_{\mc{L}}\colon X(K) \ra \R_{\geq 0}$ denote the corresponding height function, as defined in \cite[Section 1]{peyre1995hauteurs}.
We recall that the \textit{effective cone} $\Lambda_{\op{eff}}$ of a variety $X$ is defined as
$$\Lambda_{\op{eff}} = \{[D]\in \op{Pic}(X)\tensor_{\Z}\R: [D]\geq 0\}.$$

\begin{definition}\label{exponents a and b} Let $[K_X]$ denote the canonical divisor class. Given the above data, we define
$$ a = \inf\{t \in \R:  t[L]+ [K_X] + [D] \in \Lambda_{\op{eff}}\},$$
and we define $b$ to be the codimension of the minimal supported face of $\Lambda_{\op{eff}}$ which contains  $a[L] + [K_X] + [D]$. 
\end{definition}

\begin{conjecture}[Pieropan, Smeets, Tanimoto, V\'{a}rilly-Alvarado]\label{campana manin conjecture} Let $(X,D)$ be a klt Campana orbifold, and suppose that $-(K_X + D)$ is ample (in this case we say that the orbifold is \textit{Fano}). Assume that the set of Campana points $(\mc{X},\mc{D})(\OO_{K,S})$ is not itself thin. Then there is a thin set $\mc{T}$ of Campana $\OO_{K,S}$-points such that
$$\#\{P \in (\mc{X},\mc{D})(\OO_{K,S})\bs \mc{T}: H_{\mc{L}}(P) \leq B\} \sim c_{\textrm{PSTV-A}}B^a(\log B)^{b-1},$$
as $B \ra \infty$, where $a,b$ are as in Definition \ref{exponents a and b}, and $c_{\textrm{PSTV-A}}>0$ is an explicit constant described in \cite[Section 3.3]{pieropan2019campana}.
\end{conjecture}

\section{Proof of Theorem \ref{campana manin prediction}}\label{countexample}\label{section: Campana-Manin constant for 3 squarefuls}

In this section, we prove Theorem \ref{campana manin prediction}. We keep the notation from the previous section. We recall from Example \ref{example of campana orbifolds} that the Campana orbifold corresponding to $N_1(B)$ is the orbifold $(\PP^1, D)$ defined over $\Q$, where $D$ is the divisor $\sum_{i=0}^2 \frac{1}{2}D_i$ and 
$$D_0 = \{z_0=0\},\quad D_1= \{z_1=0\},\quad D_2 = \{z_0+z_1=0\}.$$ 
We choose $S = \{\infty\}$ for the set of bad places, and fix the obvious model of $(\PP^1,D)$ over $\Z$. We shall work with the height
\begin{equation}\label{the choice of height}H(z) = \max(|z_0|,|z_{1}|, |z_0+z_1|)\end{equation}
for $(z_0,z_1) \in \Z^2_{\prim}$ representing $z$. This choice of height corresponds to the ample line bundle $\mc{L}=\OO_{\PP^1}(1)$, metrized by the generating set $\{z_0,z_1,z_0+z_1\}$ for the global sections of $\mc{L}$. 

The degree function gives an isomorphism $\Pic(\Proj^1) \isom \Z$. Under this isomorphism, the line bundle $\mathcal{L}$ maps to $1$ and $\Lambda_{\op{eff}}$ is identified with $\R_{\geq 0}$. Since $\deg D = 3/2$ and $\deg [K_{\PP^1}] = -2$, we have 
\begin{align*}
        a &= \inf\left\{t \in \R:  t -2 +\frac{3}{2}\geq 0\right\}= \frac{1}{2}.
\end{align*}
The minimal supported face of $\Lambda_{\op{eff}}$ which contains $a[L] + [K_{\PP^1}] + [D] = 0$ is $\{0\}$, which has codimension $1$ in $\Lambda_{\op{eff}}$, and so $b = 1$. These values of $a$ and $b$ are compatible with Conjecture \ref{conjecture 1.1}. 

We now turn our attention to the definition of the leading constant $c_{\textrm{PSTV-A}}$, and its computation for the orbifold and height function corresponding to $N_1(B)$. For a description of $c_{\textrm{PSTV-A}}$ in full generality, we refer the reader to \cite[Section 3.3]{pieropan2019campana}. Here, for simplicity, we define $c_{\textrm{PSTV-A}}$ in the case when $X$ is a smooth projective variety over $\Q$ satisfying $a[L]+[K_X]+[D]=0$ (this latter hypothesis in particular holds when $\Pic(X) \isom \Z$). These assumptions certainly hold in the setting of Theorem \ref{campana manin prediction}, where we take $X=\PP^1$. The constant $c_{\textrm{PSTV-A}}$ is given by the formula
\begin{equation}\label{factors of c}c_{\textrm{PSTV-A}} = \frac{\alpha\beta\tau}{a(b-1)!},\end{equation}
and we proceed to discuss each of the factors $\alpha,\beta, \tau$ in turn. 

Let $\rho$ denote the rank of $\Pic(X)$. The \textit{dual effective cone} $\Lambda^*_{\op{eff}}$ is defined as
$$\Lambda_{\op{eff}}^{*} = \{y \in (\Pic(X)\tensor_{\Z}\R)^{*}: \langle y, r \rangle \geq 0 \textrm{ for all }r \in \Lambda_{\op{eff}}^{*}\}.$$
Here $\Lambda_{\op{eff}}^{*} \isom (\R^{\rho})^{*} = \Hom_{\R}(\R^{\rho},\R)$ is the usual vector space dual, and $\langle\cdot, \cdot \rangle$ is the tautological pairing defined by $\langle y,r \rangle = y(r)$.

The definition of $\alpha$ from \cite[Section 3.3]{pieropan2019campana} is closely related to the $\alpha$-constant from the classical Manin conjecture. In general, the definition involves a rigid effective divisor $E$ which is $\Q$-linearly equivalent to $aL + K_X + D$. However, if $(X,D)$ is any Campana orbifold with $E=0$ and we write $D=\sum_{i=0}^k \eps_i D_i$ for prime divisors $D_i$, then the definition of $\alpha$ simplifies to  
\begin{equation}\label{alpha constant}
\alpha = \prod_{i=0}^k (1-\eps_i) \int_{\Lambda^*_{\op{eff}}}e^{-\langle [L], x\rangle}\textrm{d}x.
\end{equation}

In our example, $k=2$ and $\eps_i = 1/2$ for $0\leq i \leq 2$. Since $\deg L =1$, we have $\langle [L],x\rangle = x$. Therefore
\begin{align*}
\alpha &= \left( \frac{1}{2}\right)^3\int_{0}^{\infty}e^{-x}\textrm{d}x=\frac{1}{8}.
\end{align*}

When $a[L] + [K_X] + [D] = 0$, the constant $\beta$ from \cite[Section 3.3]{pieropan2019campana} agrees with the definition of $\beta$ in Manin's conjecture. The $\beta$-constant plays no r\^{o}le in our example, since $\beta=1$ whenever  $\Pic(X_{\overline{\Q}})\isom\Z$ (see for example \cite[Definition 5.12, Remark 5.13]{jahnel2014brauer}).

Substituting $a=\frac{1}{2}, b=1, \alpha = \frac{1}{8}$ and $\beta =1$ into (\ref{factors of c}), in our example, we conclude that 
\begin{equation}\label{a factor 1/4}
c_{\textrm{PSTV-A}}= \frac{\tau}{4}.\end{equation}

We now describe the Tamagawa number $\tau$. Again, we do not give the definition in full generality, but assume for simplicity that $a[L]+[K_X]+[D]=0$. It follows from \cite[Section 3.3]{pieropan2019campana} that
\begin{equation}\label{campana tau}\tau = \int_{\mathcal{U}(\A_{\Q})}\mathbf{H}(x, 0)^{-1}\textrm{d}\tau_{X, D}.\end{equation}
We explain the notation used in this equation. In the integrand, $\mathbf{H}(x, 0)$ denotes the height of $x$ with respect to the zero line bundle, and so this is identically $1$. In \cite[Section 3.3]{pieropan2019campana}, two alternative definitions of $\mathcal{U}(\A_{\Q})$ are given. The first is as a topological closure of the Campana $\OO_{K,S}$-points of $(\mc{X},\mc{D})$ in $X(\A_{\Q})$, and the second is in terms of the Brauer--Manin pairing. In general, it is not known whether the two definitions coincide, but in our situation the definitions do agree since there is no Brauer--Manin obstruction. Finally, the measure $\tau_{X,D}$ is defined to be $H_{D}\tau_X$, where $\tau_X$ is the usual Tamagawa measure from Manin's conjecture as defined in \cite[Section 2]{peyre1995hauteurs}, and $H_D$ is a height function associated to $D$ defined as follows. We write $D = \sum_{i=0}^k \eps_i D_i$ for prime divisors $D_i$. We fix an adelic metrization on the line bundles $\OO_X(D_i)$ associated to each of the divisors $D_i$. This induces a height $H_{D_i}$ as described in \cite[D\'{e}finition 1.2]{peyre1995hauteurs}. We then define
$$H_D = \prod_{i=1}^k H_{D_i}^{\eps_i}.$$

Below, we let $p$ denote any prime (or any non-archimedean place of $\Q$), and we let $v$ denote any place of $\Q$, including the archimedean place $v=\infty$. We let $\textrm{d}x_{i,p}$ denote the usual $p$-adic measure with respect to $x_i$, and $\textrm{d}x_{i,\infty}$ denote the usual Lebesgue measure. We denote by $\mc{K}_X$ the canonical line bundle of $X$. In the notation of \cite[Section 2]{peyre1995hauteurs}, we obtain 
$$\tau_{X,D} = H_D\omega_{\infty}\prod_{p}\det(1-p^{-1}\op{Frob}_p|\Pic(\overline{X}^{I_p}))\omega_p,$$
where 
\begin{align}\label{florian formulation}
\omega_v &= \frac{|\textrm{d}z_{1,v}\wedge \cdots \wedge \textrm{d}z_{n,v}|}{\|\textrm{d}z_1 \wedge \cdots \wedge \textrm{d}z_n\|_{\mc{K}_{X},v}}.
\end{align}

We now use the definitions above to compute $\tau$ in our example. We have  $\det(1-p^{-1}\op{Frob}_p|\Pic (\overline{\PP^1}^{I_p}))=1-p^{-1}$ for all primes $p$. In fact, this is true whenever $\Pic(X_{\overline{\Q}})\isom \Z$ \cite[Remark 6.10]{jahnel2014brauer}. Fixing $i\in\{0,1,2\}$ and writing $z_2 = z_0+z_1$, we define sections $s_{D_i} = z_i$. We take the metrization on $\OO_{\PP^1}(D_i)$ obtained from pulling back the metrization on $\OO_{\PP^1}(1)$ via the obvious isomorphism $\OO_{\PP^1}(D_i) \isom \OO_{\PP^1}(1)$. Since we are using the metrization on $\OO_{\PP^1}(1)$ arising from the generating set $\{z_0,z_1,z_2\}$, we obtain
$$H_{D_i}(z) = \prod_v\|s_{D_i}(z)\|_v^{-1} = \frac{\max(|z_0|,|z_1|,|z_2|)}{|z_i|}$$
on $(\PP^1\bs D_i)(\Q)$, for $(z_0,z_1)\in \Z^2_{\prim}$ representing $z$. Therefore,
\begin{equation}\label{hd}H_D(z) = \frac{\max(|z_0|,|z_1|,|z_2|)^{3/2}}{|z_0z_1z_2|^{1/2}}\end{equation}
on the open set $(\PP^1\bs \supp(D))(\Q)$, where $\supp(D) = D_0\cup D_1 \cup D_2$. 

The property that $z \in (X\bs\supp(D))(\Q)$ is a Campana point is a local condition. In our example, it is the condition that for all primes $p$, we have 
 $$\nu_p(z_0), \nu_p(z_1), \nu_p(z_0+z_1) \neq 1$$
 for $(z_0, z_1) \in \Z^2_{\prim}$ representing $z$, where $\nu_p$ denotes the $p$-adic valuation. Let $\Omega_p$ denote the subset of $\PP^1(\Q_p)$ cut out by this local condition, and define $\Omega_{\infty}=\PP^1(\R)$. The expression (\ref{campana tau}) becomes
\begin{equation}\label{tamagawa volumes}
\tau = \sigma_{\infty}\prod_{p}(1-p^{-1})\sigma_p,
\end{equation}
where
$$\sigma_v =\int_{\Omega_v}\frac{\max(|z_0|_v, |z_1|_v, |z_0+z_1|_v)^{3/2}}{|z_0z_1(z_0+z_1)|_v^{1/2}} \textrm{d}\omega_v.$$

To compute $\sigma_v$, we use the chart $U_v = \{[t:1]: t\in \Q_v\}$, equipped with the natural maps $f_v\colon U_v\ra \Q_v$ given by $[t:1] \mapsto t$. The only point on $\Omega_v$ not in $U_v$ is $[1:0]$, which has measure zero, and so we may replace the range of integration with $\Omega_v \cap U_v$. A point $[z_0:z_1]$ on $U_v$ satisfies $t=z_0/z_1$. Let $\textrm{d}t$ denote the usual $p$-adic measure or the Lebesgue measure as appropriate. We recall that there is an isomorphism $\mc{K}_{\PP^1} \isom \OO_{\PP^1}(-2)$, which on the chart $U_v$ is given by mapping $\textrm{d}t$ to $z_1^{-2}$. Therefore, in the notation of \cite[Section 2]{peyre1995hauteurs}, we have
$$\|\textrm{d}t\|_{\mc{K}_{\PP^1},v} = \|z_1^{-2}\|_{\OO_{\PP^1}(-2),v} = \frac{|z_1|_v^{-2}}{\max(|z_0|_v,|z_1|_v,|z_0+z_1|_v)^{-2}}.$$
Recalling (\ref{florian formulation}), we obtain
\begin{align}
    H_{D,v}\omega_{v} &= \frac{\max(|z_0|_v,|z_1|_v,|z_0+z_1|_v)^{3/2}|z_1|_v^2}{|z_0z_1(z_0+z_1)|_v^{1/2}\max(|z_0|_v,|z_1|_v,|z_0+z_1|_v)^{2}}\textrm{d}t\nonumber\\
    &=\frac{\textrm{d}t}{|t(1+t)|_v^{1/2}\max(|t|_v,1,|1+t|_v)^{1/2}}.
\end{align}
 When $v=\infty$, we have $f_{v}(\Omega_v\cap U_v) = \R$. Therefore,
 \begin{align*}
     \sigma_{\infty}&=\int_{\R}\frac{\textrm{d}t}{|t(1+t)|^{1/2}\max(|t|,1,|1+t|)^{1/2}}\\
     &= \int_{-\infty}^{-1} \frac{\textrm{d}t}{|1+t|^{1/2}|t|}+\int_{-1}^0 \frac{\textrm{d}t}{|t(1+t)|^{1/2}}+\int_{0}^{\infty}\frac{\textrm{d}t}{t^{1/2}(1+t)}.
 \end{align*}
Each of these integrals is equal to $\pi$, and so we conclude that $\sigma_{\infty}=3\pi$. In the following lemma, we compute $\sigma_v$ when $v<\infty$.
\begin{lemma}\label{p-adic density} We have
$\sigma_p = 1+ p^{-1}+3p^{-3/2}$.
\end{lemma}
\begin{proof} We recall that $\Omega_p$ consists of the points $[z_0:z_1] \in \PP^1(\Q_p)$ such that $\min(\nu_p(z_0),\nu_p(z_1))=0$ and $\nu_p(z_0),\nu_p(z_1), \nu_p(z_0+z_1) \neq 1$. From this, we see that $f_p(\Omega_p \cap U_p)$ is the set of all $t \in \Q_p$ which satisfy the conditions $t,t+1 \neq 0$ and $\nu_p(t), \nu_p(1+t)\neq \pm 1$. Therefore,
\begin{equation}\label{split the integral}
\sigma_p = \int_{\substack{t \in \Q_p\\ \nu_p(t),\nu_p(1+t) \neq \pm 1}}\frac{\textrm{d}t}{|t(1+t)|_p^{1/2}\max(|t|_p,1,|1+t|_p)^{1/2}}.
\end{equation}
By the ultrametric triangle inequality, $\max(1,|t|_p,|1+t|_p)=\max(1,|t|_p)$. We now consider separately the contribution to the integral from the regions $R_1,R_2,R_3$ defined respectively by the conditions
\begin{enumerate}
    \item $\nu_p(t) \geq 2,$
    \item $\nu_p(t) =0,$
    \item $\nu_p(t)\leq -2$.
\end{enumerate}

In the region $R_1$, we have $|1+t|_p = 1$ and $\max(|t|_p,1)=1$. We recall also that for any $j\in \Z$, the $p$-adic measure of the set of $t\in \Q_p$ with $\nu_p(t) = j$ is $(1-p^{-1})p^{-j}$. Hence the contribution to (\ref{split the integral}) from $R_1$ is 
$$\int_{\substack{t \in \Q_p\\ \nu_p(t)\geq 2}}\frac{\textrm{d}t}{|t|_p^{1/2}}= \sum_{j=2}^{\infty}(1-p^{-1})p^{-j/2} = p^{-1}+p^{-3/2}.$$

In the region $R_2$, we have $\max(1,|t|_p)=1$. We further subdivide this region according to the value of $\nu_p(1+t)$, remembering that the case $\nu_p(1+t)=1$ must be excluded. We define
$$ S_j = \{t \in \Z_p^{\times}: \nu_p(1+t)=j\}.$$
When $j<0$, we have $S_j = \emptyset$. When $j=0$, the measure of $S_j$ is $1-2p^{-1}$, because $t \in S_0$ if and only if the reduction of $t$ modulo $p$ is not $0$ or $-1$. (In the case $p=2$, we have $1-2p^{-1}=0$, which is consistent with the fact that it is not possible for $t$ and $1+t$ to both be in $\Z_2^{\times}$). When $j\geq 2$, elements $t\in S_j$ are precisely elements of the form $t = -1 + s$ for some $s \in \Q_p$ with $\nu_p(s)=j$, and so $S_j$ has measure $p^{-j}(1-p^{-1})$. We conclude that the contribution to (\ref{split the integral}) from the region $R_2$ is 
$$\int_{\substack{t \in \Z_p^{\times}\\ \nu_p(1+t) \neq \pm 1}}\frac{\textrm{d} t}{|1+t|_p^{1/2}}= 1-2p^{-1}+\sum_{j=2}^{\infty}(1-p^{-1})p^{-j/2}= 1-p^{-1}+p^{-3/2}.$$

Finally, in the region $R_3$, we have $|1+t|_p =1 $ and $\max(1,|t|_p)=|t|_p$, and so we obtain a contribution from $R_3$ of 
$$\int_{\substack{t \in \Q_p\\ \nu_p(t)\leq -2}}\frac{\textrm{d}t}{|t|_p^{3/2}}= \sum_{j=2}^{\infty}(1-p^{-1})p^{-j/2} = p^{-1}+p^{-3/2}.$$

Combining the three regions, we conclude that 
$$\sigma_p = (p^{-1}+p^{-3/2})+(1-p^{-1}+p^{-3/2})+(p^{-1}+p^{-3/2})= 1+ p^{-1}+3p^{-3/2},$$ 
as required.
\end{proof}

We now complete the proof of Theorem \ref{campana manin prediction}. We recall that $c_{\textrm{PSTV-A}}=\tau/4$, and $\sigma_{\infty}=3\pi$. Together with Lemma \ref{p-adic density} and (\ref{tamagawa volumes}), this implies that
\begin{align}
c_{\textrm{PSTV-A}} &= \frac{1}{4}\sigma_{\infty}\prod_p (1-p^{-1})\sigma_p\nonumber\\
&=\frac{1}{4}\cdot 3\pi\prod_p (1+3p^{-3/2}-p^{-2}-3p^{-5/2})\label{convergence factor setup}\\
&=\frac{1}{4}\cdot 3\pi\prod_p\left(1+\frac{3p^{-3/2}}{1+p^{-1}}\right)(1-p^{-2})\nonumber.
\end{align}
Since $\prod_p(1-p^{-2}) = 1/\zeta(2) = 6/\pi^2$, we obtain the expression for $c_{\textrm{PSTV-A}}$ claimed in (\ref{cpstva}). 

In order to estimate the numerical value of $c_{\textrm{PSTV-A}}$, we evaluate the Euler product $\prod_p (1-p^{-1})\sigma_p$ by removing convergence factors. Using (\ref{convergence factor setup}) we have
\begin{align*}
    \prod_{p}(1-p^{-1})\sigma_p &= \prod_p (1+3p^{-3/2}-p^{-2}-3p^{-5/2})\\
    &=\zeta(3/2)^3 \cdot \frac{\zeta(4)}{\zeta(2)}\cdot \left(\frac{\zeta(5)}{\zeta(5/2)}\right)^3\prod_p f(p),
\end{align*} 
where $f(p) = 1+O(p^{-3})$ is an explicit polynomial in $p^{-1}$. The resulting Euler product now converges quickly enough to obtain an approximation for $c_{\textrm{PSTV-A}}$ accurate to eight decimal digits by taking the product over the first $1000$ primes. 

\section{Manin's conjecture for the family of conics}\label{section manin for family of conics}

In this section, we describe the alternative approach of Browning and Van Valckenborgh \cite[Section 2]{browning2012sums} for predicting the leading constant for the counting problem $N_1(B)$ from (\ref{counting problem for k squarefuls}). The counting function considered in \cite{browning2012sums} is given by 
\begin{align*}
\widetilde{N_1}(B) = \#\left\{(z_0,z_1,z_2)\in \N^3_{\prim}:
\begin{tabular}{l}
$z_0+z_1=z_2,\,\, z_0,z_1,z_2 \leq B,$\\ 
$z_0,z_1,z_2 \textrm{ squareful}$
\end{tabular}
\right\}.
\end{align*}
This is very similar to $N_1(B)$, the only differences being the presence of the factor $1/2$ in (\ref{counting problem for k squarefuls}), and that in $\widetilde{N_1}(B)$ we require $(z_0,z_1,z_2)\in \N^3_{\prim}$, whilst in $N_1(B)$ we only require $(z_0,z_1,z_2) \in (\Z_{\neq 0})^3_{\prim}$. The following lemma compares $N_1(B)$ with $\widetilde{N_1}(B)$. 

\begin{lemma}\label{count comparison}
We have $N_1(B) = 3\widetilde{N_1}(B)$.
\end{lemma}

\begin{proof}
For convenience, we use the notation $\widetilde{S_1}(B)$ to mean the set which $\widetilde{N_1}(B)$ enumerates. For $\bmeps \in \{\pm 1\}^3$, we define
\begin{align*}
S_{\bm{\eps}}(B)= \left\{(z_0,z_1,z_2)\in (\Z_{\neq 0})^3_{\prim}:
\begin{tabular}{l}
 $z_0+z_1=z_2, |z_i| \leq B \textrm{ for all }i$\\ $z_i\textrm{ squareful, } \op{sgn}(z_i) = \eps_i \textrm{ for all }i$
\end{tabular}
\right\},
\end{align*}
and $N_{\bmeps}(B)=\#S_{\bmeps}(B)$. Then 
\begin{equation}\label{the sum}
    2N_1(B) = \sum_{\bm{\eps} \in \{\pm 1\}^3}N_{\bm{\eps}}(B).
\end{equation}

For $\bmep = (1,1,-1)$ or $\bmeps = (-1,-1,1)$, we have $N_{\bmeps}(B) = 0$. For $\bmeps=(1,1,1)$ or $\bmeps =(-1,-1,-1)$, we have $N_{\bmeps}(B) = \widetilde{N_1}(B)$, and so these choices of $\bmeps$ contribute $2\widetilde{N_1}(B)$ to the sum in (\ref{the sum}).

For the remaining four choices of $\bmeps$, it can be checked that there is a permutation $\sigma \in S_3$ such that the map 
\begin{align*}
S_{\bmeps}(B) &\ra \widetilde{S_1}(B)\\
(z_0,z_1,z_2) &\mapsto \sigma(|z_0|,|z_1|,|z_2|)
\end{align*}
is a bijection. Therefore $N_{\bmeps}(B) = \widetilde{N_1}(B)$, and these choices of $\bmeps$ contribute $4\widetilde{N_1}(B)$ to the sum in (\ref{the sum}).
\end{proof}

In the remainder of this section we record the explicit description of $c_{\textrm{BV}}$ from \cite{browning2012sums}, and define some notation which will be useful later. 

Recalling the discussion in the introduction, for a fixed $\by=(y_0,y_1,y_2)$ in  $(\Z_{\neq 0})^3$, we consider the conic $C_{\by}$ defined by the polynomial
\begin{equation}\label{conic definition}
F_{\by}(x_0,x_1,x_2) = y_0^3x_0^2+y_1^3x_1^2 - y_2^3x_2^2.
\end{equation}
We define an anticanonical height $H_{\by}$ on $C_{\by}$ given by 
\begin{equation}\label{Hyght}
H_{\by}(x) = \max \left(|y_0^3x_0^2|, |y_1^3x_1^2|, {|y_2^3x_2^2|}\right)^{1/2},
\end{equation}
where $(x_0,x_1,x_2) \in (\Z_{\neq 0})^3_{\prim}$ represents the point $x\in C_{\by}(\Q)$. We define 
$$N_{C_{\by},H_{\by}}(B^{1/2}) = \#\{x \in C_{\by}(\Q): H_{\by}(x) \leq B^{1/2}\},$$
and $N^{+}_{C_{\by},H_{\by}}(B^{1/2})$ in the same way, but with the additional coprimality condition $\gcd(x_0y_0,x_1y_1, x_2y_2) = 1$. Then
\begin{equation}\label{fibre expression}
\widetilde{N_1}(B) = \frac{1}{4}\sum_{\by \in \N^3}\mu^2(y_0y_1y_2)N^{+}_{C_{\by},H_{\by}}(B).
\end{equation}
The presence of the factor $1/4$ in (\ref{fibre expression}) is due to the fact that in $N^{+}_{C_{\by},H_{\by}}(B^{1/2})$ the points $x$ we count lie in $\PP^2(\Q)$, which allows for four choices of sign for the coordinates of $x$ corresponding to each point $(z_0,z_1,z_2)$ enumerated by $\widetilde{N_1}(B)$. 

As mentioned in \cite[Section 3]{browning2012sums}, it is easy to show that there is an absolute constant $\delta >0$ and an explicit constant $c_{H_{\by}}(C_{\by}(\A_{\Q})^+)$ depending on $\by$ such that
\begin{equation}\label{fibres}N^{+}_{C_{\by},H_{\by}}(B^{1/2}) = c_{H_{\by}}(C_{\by}(\A_{\Q})^+)B^{1/2}(1+O_{\by}(B^{-\delta})),
\end{equation}
where the error term has at worst polynomial dependence on $\by$. The constant $c_{H_{\by}}(C_{\by}(\A_{\Q})^+)$ is a special case of the constant conjecturally formulated by Peyre \cite[D\'{e}finition 2.5]{peyre1995hauteurs}. Here, $C_{\by}(\A_{\Q})^{+}$ denotes the open subset of $C_{\by}(\A_{\Q})$ given by the conditions $\min_{0\leq i \leq 2}(\nu_p(x_iy_i)) =0$ for all primes $p$, and is intended to reflect the coprimality condition $\gcd(x_0y_0,x_1y_1,x_2y_2)=1$ imposed on $N^+_{C_{\by},H_{\by}}(B^{1/2})$ in (\ref{fibre expression}). The computation of $c_{H_{\by}}(C_{\by}(\A_{\Q})^+)$ then involves the Tamagawa measure of $C_{\by}(\A_{\Q})^+$ in place of the full adelic space $C_{\by}(\A_{\Q})$. In the light of (\ref{fibre expression}), it is natural to predict that
\begin{equation}\label{fibre sum conjecture}
    \widetilde{N_1}(B) \sim c_{\textrm{BV}}B^{1/2},
\end{equation}
with
\begin{equation}\label{fibre sum}
c_{\textrm{BV}}= \frac{1}{4}\sum_{\by \in \N^3}\mu^2(y_0y_1y_2)c_{H_{\by}}(C_{\by}(\A_{\Q})^+).
\end{equation}
In what follows, we shall use for brevity the notation 
\begin{equation}\label{starred gamma}
\gamma(d) \colonequals \prod_{\substack{p\mid d\\ p>2}}\left(1+\frac{1}{p}\right)^{-1}.
\end{equation}
In \cite[Section 2]{browning2012sums}, it is established that
\begin{equation}\label{Manin constant on fibre}
c_{H_{\by}}(C_{\by}(\A_{\Q})^+)=\frac{4}{\pi}\cdot\frac{\mu^2(y_0y_1y_2)\gamma(y_0y_1y_2)}{(y_0y_1y_2)^{3/2}}\sigma_{2,\by}\rho(\by),
\end{equation}
where 
\begin{align}\rho(\by)&=\prod_{\substack{p\mid y_0\\p>2}}\left(1+\left(\frac{y_1y_2}{p}\right)\right)\prod_{\substack{p\mid y_1\\p>2}}\left(1+\left(\frac{y_0y_2}{p}\right)\right)\prod_{\substack{p\mid y_2\\p>2}}\left(1+\left(\frac{-y_0y_1}{p}\right)\right),\label{definition of rho(y)}\\
\sigma_{2,\by} &= \lim_{r \ra \infty}2^{-2r}\#\left\{\bx \in (\Z/2^r\Z)^3: 
\begin{tabular}{l}
$y_0^3x_0^2+y_1^3x_1^2 \equiv y_2^3x_2^2 \;\Mod{2^r},$\\
$\min_{0\leq i \leq 2}(\nu_2(x_iy_i))=0$
\end{tabular}
\right\}\label{definition of sigma 2y}.
\end{align}
Combining with (\ref{fibre sum}), we conclude that
\begin{equation}\label{the second prediction}
c_{\textrm{BV}}=\frac{1}{\pi}\sum_{\by \in \N^3}\frac{\mu^2(y_0y_1y_2)\gamma(y_0y_1y_2)}{(y_0y_1y_2)^{3/2}}\sigma_{2,\by}\rho(\by).
\end{equation}

From \cite[Lemma 2.2]{browning2012sums}, we have the calculation
\begin{equation}\label{alternative expression for sigma 2y}
\sigma_{2,\by} = 
\begin{cases}
1, & \textrm{if } 2\nmid y_0y_1y_2 \textrm { and }\neg\{y_0\equiv y_1 \equiv -y_2 \Mod{4}\},\\
2, & \textrm{if } 2\mid y_0 \textrm{ and } y_1 \equiv y_2 \Mod{8},\\
2, & \textrm{if } 2\mid y_1 \textrm{ and } y_0 \equiv y_2 \Mod{8},\\
2, & \textrm{if } 2\mid y_2 \textrm{ and } y_0 \equiv -y_1 \Mod{8},\\
0, & \textrm{otherwise.}
\end{cases}
\end{equation}
As a consequence of quadratic reciprocity, it can be shown that the condition $\neg\{y_0\equiv y_1 \equiv -y_2 \Mod{4}\}$ is automatically satisfied whenever $\rho(\by)\neq 0$.

\begin{remark} The expression for $c_{\textrm{BV}}$ given in (\ref{the second prediction}) is a sum of products of local densities arising from Manin's conjecture, but it is not multiplicative in $\by$, and it does not appear possible to express $c_{\textrm{BV}}$ as a single Euler product. This is in contrast to the constant $c_{\textrm{PSTV-A}}$, which is defined as a product of local densities. 
\end{remark}

\section{Thin sets}\label{sec:thin sets}
In this section, we prove Theorem \ref{you can get anything}. We recall the definition of the set of Campana points $\mc{C}$ from (\ref{campana points for three squarefuls}) and the corresponding counting problem $N_1(B)$ from (\ref{counting problem for k squarefuls}), with the height $H$ as defined in (\ref{the choice of height}). From Definition \ref{what are campana thin sets}, the Campana thin subsets of $\mc{C}$ take the form $\mc{T}= T \cap \mc{C}$, where $T$ is a thin subset of $\PP^1(\Q)$. For a set $S\subseteq \PP^1(\Q)$, we define $N_1(S,B) = \#\{z \in S: H(z) \leq B\}$. In particular, we have $N_1(\mc{C},B)=N_1(B)$. 

For fixed integers $y_0,y_1,y_2$ satisfying $\mu^2(y_0y_1y_2)=1$, we recall that $C_{\by}$ denotes the conic $y_0^3x_0^2+y_1^3x_1^2=y_2^3x_2^2$. Consider the morphism 
\begin{align*}
    \varphi_{\by}\colon C_{\by} &\ra \PP^1,\\
    [x_0:x_1:x_2] &\mapsto [y_0^3x_0^2: y_1^3x_1^2].
\end{align*}
The image $T_{\by}\colonequals\varphi_{\by}(C_{\by})$ is a thin subset of $\PP^1(\Q)$. Therefore $T_{\by}\cap \mathcal{C}$ is a thin set of Campana points. Explicitly, $T_{\by}\cap \mathcal{C}$ is described by the set
\begin{align}
\left\{[z_0:z_1] \in \PP^1(\Q): 
\begin{tabular}{l}
$(z_0,z_1)\in \Z^2_{\prim}, z_0,z_1,z_0+z_1 \neq 0, $\\
$(z_0,z_1,z_0+z_1) = (y_0^3x_0^2, y_1^3x_1^2,y_2^3x_2^2)\textrm{ for }x_0,x_1,x_2 \in \Z$
\end{tabular}
\right\}.
\end{align}
Since $\gcd(z_0,z_1)=1$ if and only if $\gcd(z_0,z_1,z_0+z_1)=1$, we may replace the condition $\gcd(z_0,z_1)=1$ with $\gcd(x_0y_0,x_1y_1, x_2y_2) = 1$. Hence if $\by \in \N^3$, then $N_1(T_{\by}\cap \mc{C},B)$ is just the quantity $\frac{1}{4}N^{+}_{C_{\by},H_{\by}}(B^{1/2})$ considered in Section \ref{section manin for family of conics}. For $\by \in \N^3$ satisfying $\mu^2(y_0y_1y_2)=1$, we define thin sets
$$T'_{\by} = \bigcup_{\substack{\bw \in (\Z_{\neq 0})^3\\|w_i| = y_i \textrm{ for all } i}}T_{\bw}.$$
By the arguments from Lemma \ref{count comparison}, we have $N_1(T'_{\by}\cap \mc{C},B) = 3N_1(T_{\by}\cap\mc{C},B)$. To summarise, we have a disjoint union 
$$\mc{C} = \bigcup_{\substack{\by \in \N^3\\ \mu^2(y_0y_1y_2)=1}}(T'_{\by}\cap \mc{C}),$$
where from (\ref{fibres}) and (\ref{Manin constant on fibre}), each set appearing in this union satisfies
$$N_1(T'_{\by}\cap \mc{C},B) = \frac{3}{4}N^{+}_{C_{\by},H_{\by}}(B^{1/2})\sim\frac{3}{\pi}\left(\frac{\gamma(y_0y_1y_2)}{(y_0y_1y_2)^{3/2}}\sigma_{2,\by}\rho(\by)\right)B^{1/2}.$$

For a large integer $M$, we define
\begin{equation}\label{getting a tiny constant by removing a thin set}
\mc{T}_M = \bigcup_{\substack{\by \in \N^3\\\mu^2(y_0y_1y_2)=1\\y_0,y_1,y_2\leq M}}(T'_{\by}\cap \mathcal{C}).
\end{equation}
This is a thin set of Campana points, because it is a finite union of the thin sets $T'_{\by}\cap \mc{C}$. We now assume Conjecture \ref{conjecture 1.1} holds, namely that $N_1(B)\sim 3c_{\textrm{BV}}B^{1/2}$. We deduce that 
\begin{align*}
    \frac{N_1(\mc{C}\bs \mc{T}_M, B)}{B^{1/2}} &= \frac{N_1(B) -N_1(\mc{T}_M, B)}{B^{1/2}} \\
    &\sim 3c_{\textrm{BV}}- \frac{3}{\pi}\sum_{\substack{\by \in \N^3\\ y_0,y_1,y_2 \leq M}}\frac{\mu^2(y_0y_1y_2)\gamma(y_0y_1y_2)}{(y_0y_1y_2)^{3/2}}\sigma_{2,\by}\rho(\by)\\
    &= \frac{3}{\pi}\sum_{\substack{\by \in \N^3\\\max(y_0,y_1,y_2)>M}}\frac{\mu^2(y_0y_1y_2)\gamma(y_0y_1y_2)}{(y_0y_1y_2)^{3/2}}\sigma_{2,\by}\rho(\by).
\end{align*}
Since the sum is convergent, this quantity tends to zero as $M \ra \infty$. Therefore, we have shown that we can obtain an arbitrarily small positive constant by removing a thin set. We can now complete the proof of Theorem \ref{you can get anything}.

\begin{proof}[Proof of Theorem \ref{you can get anything}]
We fix $\lambda \in (0,3c_{\textrm{BV}}]$. For a subset $S\subset \mc{C}$, we define 
$$S(B) = \{z \in S: H(z)\leq B\},$$
so that $\#S(B)=N_1(S,B)$ in our earlier notation. We require a Campana thin subset $\mc{T}\subseteq \mc{C}$ with $\#\mc{T}(B)\sim (3c_{\textrm{BV}}-\lambda) B^{1/2}$.

For an appropriate choice of $M$, the thin set $\mc{T}_M$ defined in (\ref{getting a tiny constant by removing a thin set}) satisfies $\#\mc{T}_M(B) \sim (3c_{\textrm{BV}}-\lambda_0) B^{1/2}$ for some $\lambda_0 \leq \lambda$. By definition, any subset of $\mc{T}_M$ is also thin. Therefore, it suffices to find a subset $\mc{T}\subseteq \mc{T}_M$ such that  
\begin{equation}\label{desired density}\frac{\#\mc{T}_M(B)}{\#\mc{T}(B)} \sim \frac{3c_{\textrm{BV}}-\lambda_0}{3c_{\textrm{BV}}-\lambda}.
\end{equation}
To achieve this, we take any subset $A\subseteq \N$ of the desired asymptotic density 
$$\frac{A\cap [1,B]}{B} \sim \frac{3c_{\textrm{BV}}-\lambda_0}{3c_{\textrm{BV}}-\lambda}.$$
We enumerate the elements of $\mc{T}_M(B)$ by writing $\mc{T}_M(B) = \{t_1,t_2,\ldots, t_R\}$, with $H(t_i)\leq H(t_j)$ whenever $i\leq j$. Then the set
$$\mc{T} = \{t_i \in \mc{T}_M: i \in A\}$$
is thin and satisfies (\ref{desired density}), as required.
\end{proof}

\section{Squareful values of binary quadratic forms}\label{section: squareful values of binary quadratic forms}

In this section, we study the constant $c_{\textrm{PSTV-A}}$ for an orbifold corresponding to squareful values of the binary quadratic form $ax^2+by^2$. We recall the setup from the introduction. Throughout this section, $a$ and $b$ denote positive integers satisfying $\mu^2(ab)=1$ and $a,b\equiv 1 \Mod{4}$. We consider the Campana orbifold $(X,D)$ over $\Q$, where $X=\PP^1$ and $D$ is the divisor $\frac{1}{2}V(ax^2+by^2)$, with the obvious good integral model $(\mc{X},\mc{D})$. The set of Campana points in this example is not itself thin, as can be seen by combining \cite[Theorem 1.1]{samthincampana} and \cite[Proposition 3.15]{samthincampana}. Hence the PSTV-A conjecture applies to this orbifold. We take the naive height $H$ on $\PP^1$, which is given by $H([x:y])=\max(|x|,|y|)$ for $(x,y)\in \Z^2_{\prim}$. The resulting counting problem is 
$$ N(B)= \frac{1}{2}\#\left\{(x,y) \in \Z^2_{\prim}: |x|,|y|\leq B, ax^2+by^2 \textrm{ squareful}\right\}.$$

This section is organized as follows. In Section \ref{what campana manin predicts}, we compute $c_{\textrm{PSTV-A}}$ for the orbifold $(\mc{X},\mc{D})$ and the height $H$. In Section \ref{proof of the asymptotic formula for N(B)}, we prove the asymptotic formula for $N(B)$ given in Theorem \ref{asymptotic for N(B) intro}. Finally, in Section \ref{comparing the leading constants}, we prove Corollary \ref{binary quadratic counterexample} by comparing the constants obtained in Sections \ref{what campana manin predicts} and \ref{proof of the asymptotic formula for N(B)}.

\subsection{The constant from the PSTV-A conjecture}\label{what campana manin predicts} 
The aim of this section is to prove the following theorem. We recall the notation $\gamma(n)$ from (\ref{starred gamma}).
\begin{theorem}\label{binary form campana constant}
For the orbifold and the height function defined above, the constant $c_{\textrm{PSTV-A}}$ is equal to 
$$\frac{4\gamma(ab)}{\pi^2}\left(\frac{\op{sinh}^{-1}\left(\sqrt{a/b}\right)}{\sqrt{a}}+\frac{\op{sinh}^{-1}\left(\sqrt{b/a}\right)}{\sqrt{b}}\right)\prod_{p \nmid 2ab}\left(1+ \frac{1+\left(\frac{-ab}{p}\right)}{(1+p^{-1})p^{3/2}}\right).$$
\end{theorem}

To prove Theorem \ref{binary form campana constant}, we follow the framework from Section \ref{countexample}. We keep the convention from Section \ref{section: Campana-Manin constant for 3 squarefuls} that $p$ ranges over all primes, and $v$ ranges over all primes and $v=\infty$. We have $\alpha =1/2$ and $\beta=1$, and so $c_{\textrm{PSTV-A}}= \tau/2$. The divisor $V(ax^2+by^2)$ on $\PP^1$ has degree 2, and corresponds to the line bundle $\OO_{\PP^1}(2)$. With the usual metrization, this line bundle determines the height function $\max(|x^2|,|y^2|)$ for $(x,y)\in \Z^2_{\prim}$. Choosing the section $ax^2+by^2$, we obtain 
$$ H_D  = \prod_{v} H_{D,v},$$
where
$$H_{D,v}=\frac{\max(|x|_v,|y|_v)}{|ax^2+by^2|_v^{1/2}}.$$
We use the chart $y\neq 0$, and take $z=x/y$. Then for any prime $p$, we have  $\nu_p(az^2+b)=\nu_p(ax^2+by^2)-2\nu_p(y)$. Consequently, the local Campana condition that $\nu_p(ax^2+by^2)-2\min(\nu_p(x),\nu_p(y)) \neq 1$ is equivalent to the condition that $\nu_p(az^2+b)$ is not equal to $1$ or a negative odd integer. Below, we denote by $\Omega_p$ the set of elements $z\in \Q_p$ satisfying this local Campana condition, and we set $\Omega_{\infty}=\R$. We let $\textrm{d}z$ denote the usual $p$-adic measure or the Lebesgue measure, as appropriate. We obtain
\begin{equation}\label{euler product for cpstva for N(B)}
c_{\textrm{PSTV-A}} = \frac{1}{2}\sigma_{\infty}\prod_{p}(1-p^{-1})\sigma_p,
\end{equation}
where
\begin{equation}\label{p adic factor}
\sigma_v = \int_{\Omega_v}\frac{\textrm{d}z}{\max(|z|_v,1)|az^2+b|_v^{1/2}}.
\end{equation}

To compute $\sigma_{\infty}$, we divide into regions $|z|\leq 1$ and $|z|>1$. This yields
\begin{equation}\label{arcsinhes}
\begin{split}
    \sigma_{\infty}&= \int_{|z|\leq 1} \frac{\textrm{d}z}{(az^2+b)^{1/2}}+\int_{|z|> 1}\frac{\textrm{d}z}{|z|(az^2+b)^{1/2}}\\
    &=2\left(\frac{\op{sinh}^{-1}\left(\sqrt{a/b}\right)}{\sqrt{a}}+\frac{\op{sinh}^{-1}\left(\sqrt{b/a}\right)}{\sqrt{b}}\right).
\end{split}
\end{equation}
\begin{lemma}\label{regions and cases}We have 
$$\sigma_{p} =\begin{cases}
1+p^{-1} + \left(1+\left(\frac{-ab}{p}\right)\right)p^{-3/2}, &\textrm{ if }p\nmid 2ab\\   
1, &\textrm{ if }p\mid 2ab.\\
    \end{cases}$$
\end{lemma}

\begin{proof}
We split $\Omega_p$ into three regions $R_1,R_2,R_3$, defined respectively by the conditions
\begin{enumerate}
    \item $\nu_p(z) \geq 1$,
    \item $\nu_p(z) <0$,
    \item $\nu_p(z) =0$.
\end{enumerate}
We also divide into four cases $p\nmid 2ab$, $p|a$, $p\mid b$, and $p=2$. We let $\mu_p$ denote the usual $p$-adic measure.\medskip

\nid\textit{Case 1. $p\nmid 2ab$}: 
On $R_1$, we have $|az^2+b|_p = 1$ and $\max(|z|_p,1)=1$, so
$$\int_{R_1}\frac{\textrm{d}z}{\max(|z|_p,1)|az^2+b|_p^{1/2}}=\int_{\substack{z \in \Q_p\\ \nu_p(z)\geq 1}}1\, \textrm{d}z = p^{-1}.$$

On $R_2$, we have $|az^2+b|_p = |z|_p^2$ and $\max(|z|_p,1)=|z|_p$, so we obtain a contribution of 
\begin{align*}
\int_{\substack{z \in \Q_p\\ \nu_p(z)<0}}\frac{\textrm{d}z}{|z|_p^2}&= \sum_{j=-\infty}^{-1}p^{2j}\mu_p(\{z \in \Q_p: \nu_p(z) = j\})\\
&= \sum_{j=1}^{\infty} (1-p^{-1})p^{-j}\\
&=p^{-1}.
\end{align*}

On $R_3$, we have $|az^2+b|_p \leq 1$ and $\max(|z|_p,1)=1$. For $j\geq 0$, we define
\begin{align*}
    f(j) &= \mu_p(\{z \in \Z_p^{\times}: \nu_p(az^2+b) = j\}),\\
    g(j) &= \mu_p(\{z \in \Z_p^{\times}: \nu_p(az^2+b) \geq j\}).
\end{align*}
We have
\begin{equation}\label{region 3}
\int_{R_3}\frac{\textrm{d}z}{\max(|z|_p,1)|az^2+b|_p^{1/2}}=\int_{z \in \Z_p^{\times}\cap \,\Omega_p} \frac{\textrm{d}z}{|az^2+b|_p^{1/2}}= \sum_{\substack{j\geq 0\\j\neq 1}}p^{j/2}f(j).
\end{equation}
Clearly $f(j) = g(j)-g(j+1)$ for any $j\geq 0$. We now compute $g(j)$. We have $g(0)= \mu_p(\Z_p^{\times})= 1-p^{-1}$. By Hensel's Lemma, for $j\geq 1$, we have
\begin{align*}
    g(j) &= p^{-j}\#\{z \Mod{p^j}: az^2 \equiv -b \Mod{p^j}\}\\
    &= p^{-j}\left(1+\left(\frac{-ab}{p}\right)\right).
\end{align*}
Therefore, the right hand side of (\ref{region 3}) equals
\begin{align*}
&1-p^{-1} -p^{-1}\left(1+\left(\frac{-ab}{p}\right)\right) +\sum_{j\geq 2}(1-p^{-1})p^{-j/2}\left(1+\left(\frac{-ab}{p}\right)\right)\\
&= 1-p^{-1}+\left(1+\left(\frac{-ab}{p}\right)\right)p^{-3/2}.
\end{align*}

Combining the three regions, we have completed the proof for primes $p\nmid 2ab$.\medskip

\nid\textit{Case 2. $p\mid b$}: This time, the region $R_1$ contributes zero, because if $p\mid b$ and $\nu_p(z) \geq 1$ then $\nu_p(az^2+b) =1$ (by the assumption that $b$ is squarefree), and so $z \notin \Omega_p$. The region $R_2$ contributes $p^{-1}$ to the integral in (\ref{p adic factor}) by the same calculation as in Case 1. On the region $R_3$ we have $\nu_p(az^2+b)=0$, and so
$$\int_{R_3}\frac{\textrm{d}z}{\max(|z|_p,1)|az^2+b|_p^{1/2}}=\int_{z\in \Z_p^{\times}}1\, \textrm{d}z = 1-p^{-1}.$$ 
Hence $\sigma_p = p^{-1} + 1-p^{-1} = 1$.\medskip

\nid\textit{Case 3. $p\mid a$}: The region $R_1$ contributes $p^{-1}$ by the same calculation as in Case 1. On $R_2$, we have $\nu_p(az^2+b)=2\nu_p(z)+1$, which is an odd negative integer, and so the contribution is zero. On $R_3$, we have $\nu_p(az^2+b)=0$ (since $p\nmid b$ by the assumptions $p\mid a$ and $\mu^2(ab)=1$), and so we obtain a contribution of $1-p^{-1}$ as in Case 2. Combining, we have $\sigma_p =p^{-1}+1-p^{-1} = 1.$\medskip

\nid\textit{Case 4. $p=2$}: Regions $R_1$ and $R_2$ contribute $p^{-1}$ to the integral in (\ref{p adic factor}) as in Case 1. The region $R_3$ contributes zero. To see this, we note that if $z \in \Z_2^{\times}$, then $z^2 \equiv 1 \Mod{4}$. However, since $a,b\equiv 1 \Mod{4}$, we have $az^2+b \equiv 2 \Mod{4}$, and hence $\nu_p(az^2+b) =1$. Hence $\sigma_2 = 2^{-1}+2^{-1} = 1$.
\end{proof}
Let $\sigma_{\infty}$ be as given in (\ref{arcsinhes}). We conclude from (\ref{euler product for cpstva for N(B)}) and Lemma \ref{regions and cases} that
\begin{align}
    c_{\textrm{PSTV-A}} &= \frac{\sigma_{\infty}}{2}\prod_{p \nmid 2ab}(1-p^{-1})\left(1+p^{-1} + \left(1+\left(\frac{-ab}{p}\right)\right)p^{-3/2}\right)\prod_{p\mid 2ab}(1-p^{-1})\nonumber\\
    &=\frac{\sigma_{\infty}}{2}\cdot \frac{6}{\pi^2}\prod_{p\nmid 2ab}\left(1+ \frac{1+\left(\frac{-ab}{p}\right)}{(1+p^{-1})p^{3/2}}\right)\prod_{p\mid 2ab}\frac{1}{1+p^{-1}}\nonumber\\
    &=\frac{2\sigma_{\infty}\gamma(ab)}{\pi^2}\prod_{p \nmid 2ab}\left(1+ \frac{1+\left(\frac{-ab}{p}\right)}{(1+p^{-1})p^{3/2}}\right)\label{campana manin constant for norm form}.
\end{align}
This completes the proof of Theorem \ref{binary form campana constant}.\medskip 

\subsection{The asymptotic formula for $N(B)$}\label{proof of the asymptotic formula for N(B)}
In this section, we prove Theorem \ref{asymptotic for N(B) intro}. We write $ax^2+by^2 = u^2v^3$ for $v \in \Z_{\neq 0}$ squarefree and $u\in \N$. If $\gcd(x,y)=1$, then $\gcd(a,v)=1$ and $\gcd(b,v)=1$. This is because if $p\mid \gcd(a,v)$ then $p\mid by^2$, and since $\gcd(a,b)=1$, this implies that $p\mid y$. But then $p^2\mid ax^2$, and since $a$ is squarefree, we have $p\mid x$. This contradicts the assumption $\gcd(x,y)=1$. The argument to show that $\gcd(b,v)=1$ is the same by symmetry. Hence $a,b$ and $v$ are squarefree and pairwise coprime, in other words $\mu^2(abv)=1$. Moreover, the assumptions $a,b>0$ imply that $v>0$. Therefore, we have 
$$N(B) = \frac{1}{2}\sum_{v=1}^{\infty}\mu^2(abv)N_v(B),$$
where 
$$N_v(B) = \frac{1}{2}\#\left\{(x,y,u) \in \Z^3: \gcd(x,y)=1, |x|,|y|\leq B, ax^2+by^2=u^2v^3\right\}.$$
The factor $1/2$ comes from the fact that there are two choices for the sign of $u$ in $[x:y:u]$ corresponding to each point $[x:y]$ enumerated by $N(B)$. 

Throughout this section, all implied constants depend only on $a,b$ and $\eps$. We split the sum over $v$ into ranges $v< B^{\delta}$ and $v\geq B^{\delta}$, for a fixed $\delta >0$. To deal with the range $v\geq B^{\delta}$, we note that $ax^2+by^2=u^2v^3$ and $|x|,|y|\leq B$ together imply that $u^2v^3 \ll B^2$, so $u\ll Bv^{-3/2}$. Therefore, there are $O(Bv^{-3/2})$ choices for $u$. Applying a result of Browning and Gorodnik \cite[Theorem 1.11]{browning2017power}, for any fixed $u,v$, we have
$$\#\left\{(x,y) \in \Z^2_{\prim}: ax^2+by^2 = u^2v^3\right\} = O(B^{\eps}).$$
Hence $N_v(B) \ll B^{1+\eps}v^{-3/2}$. Taking a sum over $v\geq B^{\delta}$, we obtain
\begin{equation}\label{large values of v}
\sum_{v\geq B^{\delta}}\mu^2(abv)N_v(B) \ll B^{1+\eps-\delta/2},\end{equation}
and so the contribution from the range $v\geq B^{\delta}$ is negligible.  

For the range $v<B^{\delta}$, we view the equation $ax^2+by^2=u^2v^3$ as a conic, with $a,b$ and $v$ fixed. Sofos \cite{sofos2014uniformly} counts rational points on isotropic conics by using a birational map from the conic to $\PP^1$ in order to parameterise the solutions as lattice points. Unfortunately, we cannot apply \cite[Theorem 1.1]{sofos2014uniformly} directly, since the coprimality condition $\gcd(x,y,u)=1$ is used instead of $\gcd(x,y)=1$. However, the argument can be adapted to deal with this alternative coprimality condition. We summarise the main alterations required.

Let $Q$ be a non-singular quadratic form in $3$ variables with integer coefficients. Let $\Delta_{Q}$ denote the discriminant of $Q$, and $\langle Q \rangle$ the maximum modulus of the coefficients of $Q$. Suppose that $\|\cdot \|$ is a norm isometric to the supremum norm. For convenience, below we use variables $\bx = (x_1,x_2,x_3)$ in place of $(x,y,u)$. We define 
$$N_{\|\cdot\|}(Q,B) = \#\{\bx \in \Z^3: \gcd(x_1,x_2)=1, Q(\bx)=0, \|\bx\| \leq B\}.$$
This is the same as the counting function from \cite{sofos2014uniformly}, but with the condition $\gcd(x_1,x_2)=1$ instead of $\gcd(x_1,x_2,x_3)=1$. We let $Q_v(\bx)=ax_1^2+bx_2^2-v^3x_3^2$, and define a norm $\|\cdot\|$ by $\|\bx\| = \max(|x_1|,|x_2|)$. There is a constant $C$, depending only on $a$ and $b$, such that $\|\bx\| = \max(|x_1|,|x_2|,Cv^{3/2}|x_3|)$, and so $\|\cdot\|$ is isometric to the supremum norm. In our earlier notation, we have $N_v(B) = N_{\|\cdot \|}(Q_v,B)$.

As in \cite[Section 6]{sofos2014uniformly}, the first stage is to apply a linear change of variables in order to transform $Q_v$ into a quadratic form $Q$ satisfying $Q(0,1,0)=0$. We assume that $N_v(B) > 0$, so that there exists $(t_{12},t_{22},t_{32}) \in \Z^3$ with $Q_v(t_{12},t_{22},t_{32}) = 0$ and $\gcd(t_{12},t_{22}) = 1$; we shall choose the smallest such solution. Then we can find integers $t_{11}, t_{21}$ such that $t_{11}t_{22}-t_{21}t_{12} = 1$ and $|t_{11}|,|t_{21}| \leq \max(|t_{12}|,|t_{22}|)$. Let 
$$M=\begin{pmatrix}
t_{11} & t_{12} & 0\\
t_{21} & t_{22} & 0\\
0 & t_{32} & 1
\end{pmatrix}.$$

We define $Q(\bx) = Q_v(M\bx)$ and $\|\bx\|' = \|M\bx\|$. Since the first $2\times 2$ minor of $M$ is an element of $SL_2(\Z)$, the coprimality condition $\gcd(x_1,x_2)=1$ is preserved under this transformation. Therefore $N_{\|\cdot\|}(Q_v,B) = N_{\|\cdot\|'}(Q,B)$, which we shall abbreviate to $N(Q,B)$. 

The forms $L(s,t)$ and $g(s,t)$ defined in \cite[Equation (2.3)]{sofos2014uniformly} can be written explicitly as 
\begin{align}
        L(s,t) &= (2at_{11}t_{12} + 2bt_{21}t_{22})s - 2v^3t_{32}t,\label{L(s,t)} \\
        g(s,t) &=(at_{11}^2 + bt_{21}^2)s^2 - v^3t^2.\label{g(s,t)}
\end{align}
As in \cite[Equation (2.4)]{sofos2014uniformly}, we let $\bm{q} = (q_1,q_2,q_3) = (q_1(s,t),q_2(s,t),q_3(s,t))$, where
$$ q_1(s,t) = sL(s,t), \qquad q_2(s,t) = -g(s,t),\qquad q_3(s,t) = tL(s,t).$$
By applying the parameterisation argument from \cite[Lemma 3.1]{sofos2014uniformly}, we find that  $N(Q,B) = \mc{N}(Q,B) +O(1)$, where
\begin{equation}\label{ef param}
\mc{N}(Q,B) = \#\left\{(s,t) \in \Z^2_{\prim}: t>0, \|\bm{q}\|'\leq \lambda B,\, \gcd\left(\frac{q_1}{\lambda}, \frac{q_2}{\lambda}\right) = 1\right\}
\end{equation}
and $\lambda = \gcd(q_1, q_2, q_3)$. 

We now take a sum over the possible values of $\lambda$. Due to our alternative coprimality condition, in (\ref{ef param}) we have the stronger condition $\gcd(\frac{q_1}{\lambda}, \frac{q_2}{\lambda})=1$ in place of $\gcd(\frac{q_1}{\lambda}, \frac{q_2}{\lambda}, \frac{q_3}{\lambda}) = 1$, and so when applying M\"{o}bius inversion we take a sum over a variable $r$ with $r \mid \left(\frac{q_1}{\lambda} , \frac{q_2}{\lambda}\right)$ in place of Sofos' sum over $k\mid \left(\frac{q_1}{\lambda} , \frac{q_2}{\lambda}, \frac{q_3}{\lambda}\right)$. As in \cite[Equation (3.2)]{sofos2014uniformly}, we define $$M^{*}_{\sigma, \tau}(T,n) = \#\{(s,t) \in \Z^2_{\prim}: (s,t) \equiv (\sigma, \tau) \Mod{n}, t>0, \|\bm{q}\|'\leq T\}.$$
Then similarly to \cite[Lemma 3.2]{sofos2014uniformly}, we obtain
\begin{equation}\label{mobius on nqb}
\mc{N}(Q,B) = \sum_{\lambda\mid \Delta_Q}\sum_{r} \mu(r) \sideset{}{^+}\sum_{\sigma, \tau} M^*_{\sigma, \tau}
(B\lambda, r\lambda),
\end{equation}
where $\sum^{+}$ denotes a sum over residues $\sigma,\tau$ modulo $r\lambda$ such that $\lambda \mid \bm{q}(\sigma, \tau)$, $r\lambda \mid (q_1(\sigma,\tau),q_2(\sigma,\tau))$ and $\gcd(\sigma, \tau, r\lambda)=1$. 

We now explain why, with our choice of $Q$, we may restrict the $r$-sum in (\ref{mobius on nqb}) to divisors of $\lambda$. Since $r$ is squarefree, it suffices to show that for any prime $p\mid (q_1,q_2)$, we also have $p\mid q_3$. (In general, $\gcd(q_1,q_2)$ can still be larger than $\lambda$ since its prime factors can occur with higher multiplicity.) Suppose that $p\mid (q_1,q_2)$. We immediately deduce that $p\mid q_3$ if $p\mid L(s,t)$, and so using $p \mid q_1$ we may assume that $p\mid s$. Since $\gcd(s,t)=1$ and $p\mid q_2$, we see from (\ref{g(s,t)}) that $p\mid v$. However, then from (\ref{L(s,t)}) we have that $p\mid L(s,t)$ after all, and so $p\mid q_3$, as desired. 

An asymptotic formula for $N(Q,B)$ can now be deduced by applying the lattice counting results from \cite[Section 4]{sofos2014uniformly} to estimate $M^*_{\sigma, \tau}(B\lambda,r\lambda)$. We maintain control over the resulting error terms after performing the summations in (\ref{mobius on nqb}) thanks to the restriction on the $r$-sum. Similarly to \cite[Proposition 2.1]{sofos2014uniformly}, we obtain
\begin{equation}\label{prop 2.1*}
N(Q,B) = c_{v}B + O((BK)^{1/2+\eps}(|\Delta_{Q}| + \langle Q \rangle )^{1+\eps})
\end{equation}
for some constant $c_v>0$, where
$$K = \sup_{\bx \neq \bm{0}}\left(1+ \frac{\|\bx\|_{\infty}}{\|\bx\|'} \right)$$
and $\|\bx\|_{\infty} = \max(|x_1|,|x_2|,|x_3|)$ denotes the supremum norm of $\bx$.

We have $\Delta_{Q} = \Delta_{Q_v} = abv^3 \ll v^3$. Let $\|M\|_{\infty}$ denote the maximum modulus of the entries of $M$. Then $\langle Q \rangle \ll \|M\|_{\infty}^2$. Moreover, making a change of variables from $\bx$ to $M^{-1}\bx$ in the definition of $K$, we have
$$K =\sup_{\bx \neq \bm{0}}\left(1+\frac{\|M^{-1}\bx\|_{\infty}}{\|\bx\|}\right) \ll \|M^{-1}\|_{\infty}\sup_{\bx \neq \bm{0}}\left(1+\frac{\|\bx\|_{\infty}}{\|\bx\|}\right) \ll \|M^{-1}\|_{\infty}.$$
Using the bound $\|M^{-1}\|_{\infty} \ll \|M\|_{\infty}^2$, we conclude that 
\begin{equation}\label{nqb before cassels}
N(Q,B) = c_{v}B + O((B\|M\|_{\infty}^2)^{1/2+\eps}(v^3 + \|M\|_{\infty}^2)^{1+\eps}).
\end{equation}

We recall that $\|M\|_{\infty} = \max(|t_{12}|, |t_{22}|,|t_{32}|)$ is the size of the least integral solution to $Q_v(\bx) = 0$ with $\gcd(x_1, x_2)=1$. Cassels \cite{cassels_1955} establishes an upper bound for the smallest integral solution to a quadratic form. In the following lemma, we find a bound for the least solution satisfying our additional coprimality condition. 

\begin{lemma}\label{coprime cassels}
Suppose that $a,b,v$ are integers with $\mu^2(abv)=1$. Let $Q_v$ denote the quadratic form $ax_1^2+bx_2^2-v^3x_3^2$. Then if the system
\begin{equation}\label{our system}
\begin{cases}
Q_v(\bx) =0,\\
\bx\in \Z^3,\, \gcd(x_1,x_2)=1
\end{cases}
\end{equation}
has a nontrivial solution, it has a solution satisfying $\|\bx\|_{\infty}\ll |v|^7.$
\end{lemma}

We deduce Lemma \ref{coprime cassels} from the following result of Dietmann, which generalises Cassel's argument by imposing congruence conditions on the variables. 

\begin{lemma}{\cite[Proposition 1]{dietmann2003small}}\label{Dietmann}
Let $Q$ be a non-degenerate quadratic form in $3$ variables with integral coefficients. Let $\boldsymbol{\xi}\in \Z^3$ and $\eta \in \N$. Suppose that there exists an integral solution to the system 
\begin{equation}\label{Dietmann system}
\begin{cases}
Q(\bx) =0,\\
\bx \equiv \boldsymbol{\xi} \Mod{\eta}.
\end{cases}
\end{equation}
Then there exists an integral solution to this system satisfying 
$$\|\bx\|_{\infty}\ll \max\{\eta^3|\Delta_Q|^2\langle Q \rangle ^2, \eta^3|\Delta_Q|^3\}.$$
\end{lemma}

\begin{proof}[Proof of Lemma \ref{coprime cassels}]
Suppose that $\by= (y_1,y_2,y_3)$ is a solution to (\ref{our system}). Let $Q_0$ denote the quadratic form $ax_1^2+ bx_2^2- vx_3^2$. Then clearly $Q_0(y_1,y_2,vy_3)=0$. Let $\eta = |v|$ and let $\boldsymbol{\xi} = (\xi_1,\xi_2,0)$ denote the residues of $(y_1,y_2,vy_3)$ modulo $\eta$. We have $\gcd(y_1,v)=1$, because if $p \mid (y_1,v)$ then $p \mid by_2^2$, but since $\mu^2(abv)=1$ this implies $p \mid y_2$, contradicting the assumption $\gcd(y_1,y_2)=1$. Consequently, $\xi_1$ is invertible modulo $\eta$. 

Since $\Delta_{Q_0} \ll |v|$ and $\langle Q_0 \rangle \ll |v|$, we find from Lemma \ref{Dietmann} an integral solution $\bz = (z_1,z_2,z_3)$ to (\ref{Dietmann system}) with the above choice of $Q_0, \eta, \boldsymbol{\xi}$, and with $\|\bz\|_{\infty} \ll |v|^7$. Choose $\bx = (z_1,z_2,z_3/v)/\lambda$, where $\lambda = \gcd(z_1,z_2,z_3/v)$. This is an integral solution to $Q_v =0$ because $z_3 \equiv 0 \Mod{v}$. Additionally, the bound $\|\bz\|_{\infty}\ll |v|^7$ implies that $\|\bx\|_{\infty} \ll |v|^7$. To complete the proof, it suffices to show that $\gcd(x_1,x_2)=1$, or equivalently that $\gcd(z_1,z_2) = \lambda$. Clearly $\lambda \mid \gcd(z_1,z_2)$. Conversely, suppose that $h \mid (z_1,z_2)$. From $Q_0(\bz)=0$, we see that $h \mid vz_3$. However, since $z_1 \equiv \xi_1 \Mod{\eta}$ and $\xi_1$ is invertible modulo $\eta$, we have $\gcd(h,v)=1$. Therefore, $h \mid z_3/v$, and so $h\mid \lambda$, as required.
\end{proof}

Substituting the bound $\|M\|_{\infty}\ll v^7$ from Lemma \ref{coprime cassels} into (\ref{nqb before cassels}) we conclude that 
$$N_v(B) = c_vB + O(B^{1/2+\eps}v^{21}).$$

The leading constant $c_v$ could be computed explicitly from the above method. However, we note that by \cite[Example 3.2]{peyre2018beyond}, equidistribution holds for smooth isotropic conics, and so $c_v$ is known to be the constant predicted in Manin's conjecture. More precisely, we have 
$$c_v  = \frac{1}{2}\sigma_{\infty,v}\prod_p \sigma_{p,v},$$
where $\sigma_{\infty,v}$ is the real density from Manin's conjecture applied to $N_v(B)$, and 
\begin{align*}
    \sigma_{p,v} &= \lim_{n \ra \infty}\frac{M_v(p^n)}{p^{2n}},\\
  M_v(p^n) &= \#\left\{(x,y) \Mod{p^n}: p\nmid \gcd(x,y), ax^2+by^2 \equiv u^2v^3\right\}.
\end{align*}
Combining with (\ref{large values of v}) and choosing $\delta = 1/43$, we obtain 
\begin{equation}\label{asymptotic for N(B)}
N(B) = \frac{1}{2}\sum_{v\leq B^{\delta}}\mu^2(abv)c_vB + O(B^{85/86+\eps}).
\end{equation}

We are now in a very similar situation to the one encountered in Section 4. Let $C_{(a,b,v)}$, $H_{(a,b,v)}$ and $\sigma_{2,(a,b,v)}$ be as defined in (\ref{conic definition}), (\ref{Hyght}) and (4.9) respectively, but with $(y_0^3, y_1^3, y_2^3)$ replaced by $(a,b,v^3)$. Then, analogously to (4.7), we have 
\begin{equation*}
c_{H_{(a,b,v)}}(C_{(a,b,v)}(\A_{\Q})^+)=\frac{4}{\pi}\cdot\frac{\mu^2(abv)\gamma(abv)}{(abv^3)^{1/2}}\sigma_{2,(a,b,v)}\rho(a,b,v).
\end{equation*}
The only difference between $c_v$ and $c_{H_{(a,b,v)}}(C_{(a,b,v)}(\A_{\Q})^+)$ lies in the computation of the density at the real place, since we are using a different height to $H_{(a,b,v)}$. Replacing the real density $\pi/(abv^3)^{1/2}$ appearing in \cite[Section 2.3]{browning2012sums} with the appropriate real density $\sigma_{\infty,v}$ for our setup, we have
\begin{equation}\label{real density swap}
c_v =\frac{(abv^3)^{1/2}}{\pi}\sigma_{\infty,v}c_{H_{(a,b,v)}}(C_{(a,b,v)}(\A_{\Q})^+).
\end{equation}

To compute $\sigma_{\infty,v}$, we use the Leray form as in \cite[Section 2.3]{browning2012sums} to obtain
\begin{align}
    \sigma_{\infty,v}&=\frac{1}{2v^{3/2}}\int_{[-1,1]^2}\frac{\textrm{d}x\, \textrm{d}y}{\sqrt{ax^2+by^2}}\nonumber\\
    &= \frac{1}{2v^{3/2}}\int_{-1}^1 \frac{2}{\sqrt{a}}\sinh^{-1}\left(\frac{1}{y}\sqrt{\frac{a}{b}}\right)\textrm{d}y\nonumber\\
    &=\frac{2}{v^{3/2}a^{1/2}}\left(\op{sinh}^{-1}\left(\sqrt{\frac{a}{b}}\right)+\frac{\op{sinh}^{-1}\left(\sqrt{\frac{b}{a}}\right)}{\sqrt{b/a}}\right)\nonumber\\
    &=\frac{2}{v^{3/2}}\left(\frac{\op{sinh}^{-1}\left(\sqrt{a/b}\right)}{\sqrt{a}}+\frac{\op{sinh}^{-1}\left(\sqrt{b/a}\right)}{\sqrt{b}}\right).\label{leray for N(B)}
\end{align}
Hence $\sigma_{\infty,v} = \sigma_{\infty}v^{-3/2}$, where $\sigma_{\infty}$ is the real density from the PSTV-A conjecture, as computed in (\ref{arcsinhes}). Due to the assumptions $a\equiv b \equiv 1 \Mod{4}$, the density at the prime $2$ from (\ref{alternative expression for sigma 2y}) simplifies to 
$$\sigma_{2,(a,b,v)}=\begin{cases}
1, &\textrm{ if }v \equiv 1 \Mod{4},\\
0, &\textrm{ otherwise.}
\end{cases}$$
Combining this with (\ref{real density swap}) and (\ref{Manin constant on fibre}), we conclude that 
\begin{equation}
c_v=\begin{cases}
\frac{4\sigma_{\infty}\gamma(abv)\rho(a,b,v)}{\pi^2v^{3/2}}, &\textrm{ if }v \equiv 1 \Mod{4},\\
0, &\textrm{ otherwise.}
\end{cases} 
\end{equation}
In particular, $c_v \ll v^{-3/2+\eps}$. This allows us to extend the sum in (\ref{asymptotic for N(B)}) to an infinite sum over $v$, with the same error term $O(B^{1+\eps-\delta/2})$ that is already present in (\ref{asymptotic for N(B)}). We conclude that $N(B)=cB+O(B^{85/86+\eps})$, where 
\begin{align}
c &= \frac{2\sigma_{\infty}}{\pi^2} \sum_{v \equiv 1\,(\textrm{mod }4)}\frac{\mu^2(abv)\gamma(abv)\rho(a,b,v)}{v^{3/2}}\label{manin constant for norm form}.
\end{align}
This completes the proof of Theorem \ref{asymptotic for N(B) intro}. 

\subsection{Comparison of $c$ and $c_{\op{PSTV-A}}$}\label{comparing the leading constants} Continuing from (\ref{manin constant for norm form}), we pull out a factor $\gamma(ab)$, and replace $\mu^2(abv)$ with $\mu^2(v)$ and the condition $\gcd(v,ab)=1$. This allows us to rewrite $c$ as
\begin{equation}\label{generally not multiplicative}c=R\sum_{\substack{v \equiv 1\,(\textrm{mod } 4)\\ \gcd(v,ab)=1}}\frac{\mu^2(v)\gamma(v)\rho(a,b,v)}{v^{3/2}},\end{equation}
where $R \colonequals 2\sigma_{\infty}\gamma(ab)/\pi^2$ is the same factor that appears in $(\ref{campana manin constant for norm form})$. It remains to compare the sum in (\ref{generally not multiplicative}) with the Euler product from (\ref{campana manin constant for norm form}). 

Using quadratic reciprocity, it can be shown that if $\rho(a,b,v)\neq 0$ and $v$ is odd, then $v\equiv 1 \Mod{4}$. We define 
$$f(v) = \frac{\mu^2(v)\gamma(v)}{v^{3/2}}\prod_{p\mid v}\left(1+\left(\frac{-ab}{v}\right)\right).$$
The function $f$ is multiplicative in $v$. Therefore, from (\ref{generally not multiplicative}), we have
\begin{align*}
    \frac{c}{R} &= \sum_{\gcd(v,2ab)=1}f(v)\prod_{p\mid a}\left(1+\left(\frac{bv}{p}\right)\right)\prod_{p\mid b}\left(1+\left(\frac{av}{p}\right)\right)\\
    &=\sum_{k\mid a}\sum_{l\mid b}\sum_{\gcd(v,2ab)=1}f(v)\left(\frac{bv}{k}\right)\left(\frac{av}{l}\right),
\end{align*}
where the summand is multiplicative in $v$. We conclude that 
\begin{equation}
\frac{c}{R}=\sum_{k\mid a}\sum_{l\mid b}\prod_{p \nmid 2ab}\left(\left(\frac{b}{k}\right)\left(\frac{a}{l}\right)+ \frac{\left(\frac{bp}{k}\right)\left(\frac{ap}{l}\right)\left(1+\left(\frac{-ab}{p}\right)\right)}{(1+p^{-1})p^{3/2}}\right)\label{cpstva and other things}.
\end{equation}

We recognise the contribution to (\ref{cpstva and other things}) from $k=l=1$ as precisely the Euler product $c_{\textrm{PSTV-A}}/R$ from (\ref{campana manin constant for norm form}). If $a=1$ and  $b \neq 1$, which is a special case of the norm forms considered in \cite{streeter2021campana}, then (\ref{cpstva and other things}) simplifies to 
$$\frac{c}{R}=\sum_{l\mid b}\prod_{p \nmid 2b}\left(1+ \frac{\left(\frac{p}{l}\right)\left(1+\left(\frac{-b}{p}\right)\right)}{(1+p^{-1})p^{3/2}}\right).$$
The contribution from each divisor $l$ is positive. We summarize as follows.
\begin{lemma}
Suppose that $a=1$, $\mu^2(b)=1$ and $b\equiv 1 \Mod{4}$. Then 
\begin{enumerate}[label=\roman*)]
\item $c_{\textrm{PSTV-A}}=c$ if $b=1$,
\item $c_{\textrm{PSTV-A}}<c$ if $b>1$.
\end{enumerate}
\end{lemma}

Similarly to the situation from Section \ref{sec:thin sets}, we can obtain any constant in $(0,c]$, including $c_{\textrm{PSTV-A}}$ itself, by the removal of an appropriate thin set.\medskip

Finally, we show that when $a,b>1$, it is possible that $c<c_{\textrm{PSTV-A}}$. Since the removal of thin sets can only reduce the constant $c$, this provides a counterexample to the leading constant predicted by the PSTV-A conjecture.
\begin{proof}[Proof of Corollary \ref{binary quadratic counterexample}] We take $a,b >7$ to be distinct primes satisfying the following conditions.
\begin{enumerate}
    \item $a,b \equiv 1 \Mod{4}.$
    \item $\left(\frac{a}{b}\right)=-1$.
    \item $\left(\frac{a}{p}\right) = \left(\frac{b}{p}\right)=1$ for $p\in \{3,7\}$.
    \item $\left(\frac{a}{5}\right) = -1$ and $ \left(\frac{b}{5}\right)=1$.
\end{enumerate}
The pair $a=37, b=109$ satisfies conditions $(1)$--$(4)$. In fact, $(1)$--$(4)$ are equivalent to $a,b$ lying in certain congruence classes, and so by Dirichlet's theorem on primes in arithmetic progressions, these conditions are satisfied by infinitely many pairs of distinct primes $a,b$.

Using conditions $(1)$ and $(2)$, the right hand side of (\ref{cpstva and other things}) simplifies to 
\begin{equation}\label{foursum}
\frac{c}{R} = S(\chi_0)-S(\chi_1)-S(\chi_2)+S(\chi_3),
\end{equation}
where
\begin{center}
\begin{tabular}{l l}
    $S(\chi_0) =\prod_{p \nmid 2ab}\left(1+ \frac{1+\left(\frac{-ab}{p}\right)}{(1+p^{-1})p^{3/2}}\right),$ &
    $S(\chi_1) = \prod_{p \nmid 2ab}\left(1+ \frac{\left(\frac{a}{p}\right)\left(1+\left(\frac{-ab}{p}\right)\right)}{(1+p^{-1})p^{3/2}}\right),$\\
    $S(\chi_2)= \prod_{p \nmid 2ab}\left(1+ \frac{\left(\frac{b}{p}\right)\left(1+\left(\frac{-ab}{p}\right)\right)}{(1+p^{-1})p^{3/2}}\right)$, &
    $S(\chi_3) = \prod_{p \nmid 2ab}\left(1+ \frac{\left(\frac{ab}{p}\right)\left(1+\left(\frac{-ab}{p}\right)\right)}{(1+p^{-1})p^{3/2}}\right)$.
\end{tabular}
\end{center}
Since $S(\chi_0)= c_{\textrm{PSTV-A}}/R$, it suffices to show that $S(\chi_3)-S(\chi_1)-S(\chi_2)<0$. From conditions $(3)$ and $(4)$, we have that all the Euler factors for $S(\chi_1), S(\chi_2)$ and $S(\chi_3)$ are equal to 1 for $p\leq 7$. For $p>7$, we estimate the Euler factors trivially to obtain 
\begin{equation*}
    S(\chi_3)-S(\chi_1)-S(\chi_2)\leq \prod_{p>7}\left(1+\frac{2}{(1+p^{-1})p^{3/2}}\right)-2\prod_{p>7}\left(1-\frac{2}{(1+p^{-1})p^{3/2}}\right).
\end{equation*}
Similarly to the end of Section \ref{section: Campana-Manin constant for 3 squarefuls}, we can use convergence factors to compute numerically that 
$$\prod_{p}\left(1+\frac{2}{(1+p^{-1})p^{3/2}}\right)\left(1-\frac{2}{(1+p^{-1})p^{3/2}}\right)^{-1} = 15.206698... < 16.$$
On the other hand, it can be computed that  
$$\prod_{p\leq 7}\left(1+\frac{2}{(1+p^{-1})p^{3/2}}\right)\left(1-\frac{2}{(1+p^{-1})p^{3/2}}\right)^{-1} =8.231089... > \frac{16}{2}.$$
It follows that $S(\chi_3)-S(\chi_1)-S(\chi_2)<0$, as required.
\end{proof}

\begin{remark}
In the examples considered above, the divisor $D$ does not have strict normal crossings at the primes dividing $ab$. From this point of view, it seems natural to ask whether counting Campana $\Z[1/ab]$-points instead of Campana $\Z$-points reconciles the two leading constants $c$ and $c_{\textrm{PSTV-A}}$. However, it can be checked that in this setup, by a similar argument to the proof of Corollary \ref{binary quadratic counterexample}, there are still values of $a,b$ which provide a counterexample to the PSTV-A conjecture.
\end{remark}

\nid \textbf{Acknowledgements.} The author would like to thank Damaris Schindler and Florian Wilsch for their helpful comments on the heights and Tamagawa measures used in Section 3, together with Marta Pieropan, Sho Tanimoto and Sam Streeter for providing valuable feedback on an earlier version of this paper, and Tim Browning for many useful comments and discussions during the development of this work. The author is also grateful to the anonymous referee for providing many valuable comments and suggestions that improved the quality of the paper.


\normalsize{
\bibliographystyle{siam}
\bibliography{references.bib}
}



\normalsize
\baselineskip=17pt


\end{document}